\documentclass[11pt]{amsart}

\usepackage{amsfonts,amssymb,stmaryrd,amscd,amsmath,latexsym,amsbsy}
\usepackage{amssymb}
\usepackage{amsfonts}
\usepackage{latexsym}

\newtheorem{theorem}{Theorem}[section]
\newtheorem{lemma}[theorem]{Lemma}
\newtheorem{proposition}[theorem]{Proposition}
\newtheorem{corollary}[theorem]{Corollary}
\theoremstyle{definition}
\newtheorem{definition}[theorem]{Definition}

\newtheorem{example}[theorem]{Example}

\newtheorem{conjecture}[theorem]{Conjecture}
\newtheorem{remark}[theorem]{Remark}

\newcommand{\bc}{{\overline{c}}}
\newcommand{\Res}{\text{Res}}
\newcommand{\res}{\text{res}}
\newcommand{\ind}{\text{ind}}
\newcommand{\reg}{\text{reg}}

\newcommand{\End}{\text{End}}

\newcommand{\Hom}{\text{Hom}}

\newcommand{\Ind}{\text{Ind}}
\newcommand{\Rep}{\text{Rep}}

\renewcommand{\O}{\mathcal{O}}

\newcommand{\g}{\mathfrak{g}}
\newcommand{\h}{\mathfrak{h}}

\newcommand{\ben}{\begin{enumerate}}
\newcommand{\een}{\end{enumerate}}

\newcommand{\CC}{{\mathbb{C}}}

\theoremstyle{plain}

\newtheorem*{sol}{Solution}

\theoremstyle{definition}

\theoremstyle{remark}

\newcommand{\solu}[1]{\begin{sol}{\bf (\ref{#1})}}

\begin{document}

\title{Parabolic induction and restriction functors for rational 
Cherednik algebras}

\author{Roman Bezrukavnikov}
\address{Department of Mathematics, Massachusetts Institute of Technology,
Cambridge, MA 02139, USA}
\email{bezrukav@math.mit.edu}

\author{Pavel Etingof}
\address{Department of Mathematics, Massachusetts Institute of Technology,
Cambridge, MA 02139, USA}
\email{etingof@math.mit.edu}

\maketitle

\section{Introduction}

Parabolic induction and restriction functors play an important
role in the representation theory of finite and affine Hecke
algebras. This makes it desirable to generalize them to the
setting of double affine Hecke algebras, or Cherednik algebras. 
However, a naive attempt to do so fails: the definition of
parabolic induction and restriction functors for finite and
affine Hecke algebras uses the fact that the Hecke algebra attached to 
a parabolic subgroup can be embedded into the Hecke algebra  
attached to the whole group, which is not the case in the double
affine setting. 

One of the main goals of this paper is to circumvent 
this difficulty in the case of rational Cherednik 
algebras. The price to pay is that our functors depend on an 
additional parameter, which is a point $b$ of the reflection
representation whose stabilizer is the parabolic subgroup 
at hand. The functors for different values of $b$ are isomorphic,
but not canonically, and there is nontrivial monodromy with
respect to $b$. 

More specifically, let $W$ be a finite group acting faithfully 
on a finite dimensional complex vector space $\h$. 
Let $c$ be a conjugation invariant function on the set $S$ of reflections in
$W$, and $H_c(W,\h)$ the corresponding rational Cherednik
algebra. Let ${\mathcal O}_c(W,\h)_0$ be the category 
of $H_c(W,\h)$-modules which are finitely generated over
$\CC[\h]$ and locally nilpotent under the action of $\h$. 
Let $W'\subset W$ be a parabolic subgroup, $\h'=\h/\h^{W'}$, and 
$c'$ be the restriction of $c$ to the set of reflections in
$W'$. Then we define the parabolic induction and restriction functors 
$$
\Res_b: {\mathcal O}_c(W,\h)_0\to {\mathcal O}_c(W',\h')_0,\
\Ind_b: {\mathcal O}_c(W',\h')_0\to {\mathcal O}_c(W,\h)_0, b\in
\h_{\rm reg}^{W'}.
$$
We show that these functors are exact, and the second one is
right adjoint to the first one. We also compute some of their
values, and study their dependence on $b$; this dependence is
characterized in terms of local systems with nontrivial
monodromy. In particular, we show that in the case $W'=1$,
the functor $\Res_b$ (where $b$ is a variable) is the same as the
KZ functor of \cite{GGOR}. 

As a by-product, we show that the category ${\mathcal
O}_c(W,\h)_\lambda$ of ``Whittaker'' modules
over $H_c(W,\h)$ (i.e. the category of $H_c(W,\h)$-modules, finitely generated
over $\CC[\h]$, on which $\CC[\h^*]^W$ acts with generalized
eigenvalue $\lambda\in \h^*$) is equivalent to ${\mathcal
O}_c(W_\lambda,\h)_0$.    
 
Next, we give some applications of the parabolic induction
and restriction functors. First, we give a simple proof of the
Gordon-Stafford theorem \cite{GS}, which characterizes 
the values of $c$ (for $W=S_n,\h=\Bbb C^{n-1}$) for which 
the rational Cherednik algebra is Morita equivalent to its
spherical subalgebra. In particular, we remove the condition 
$c\notin 1/2+\Bbb Z$, which was expected to be unnecessary. 
Also, we determine some values of $c$ for Coxeter groups for
which there exist finite dimensional representations of the
rational Cherednik algebra, and find the number of such
irreducible representations. Finally, we find all the irreducible
aspherical representations in category ${\mathcal O}$ of the rational
Cherednik algebra for $W=S_n$. They turn out to coincide with 
representations for $c\in (-1,0)$ which are killed by the
Knizhnik-Zamolodchikov functor, and their number for each
$c=-r/m$ ($2\le m\le n$, $0<r<m$, $(r,m)=1$) is equal to the
number of non-$m$-regular partitions of $n$. 
This confirms a conjecture of A. Okounkov and
the first author. Also, this result implies that 
the spherical Cherednik algebra $A_c(S_n)$ is simple 
if $-1<c<0$, and allows us to strengthen the main result of
\cite{BFG} about localization functors for Cherednik algebras 
in positive characteristic. 

At the end of the paper we include an appendix by the second
author, in which the techniques of this paper are applied to 
the study of reducibility of the polynomial representation 
of the trigonometric Cherednik algebra. 

{\bf Remark.} We note that the analogs of our parabolic 
induction and restriction functors in the representation theory
of semisimple Lie algebras are the translation functors between
the regular and singular category ${\mathcal O}$; see \cite{MS}. 

{\bf Acknowledgements.} The authors thank M. Geck, V. Ginzburg
I. Gordon, and R. Rouquier for useful discussions. In particular,
they are grateful to I.Gordon and R. Rouquier for reading 
preliminary versions of the paper and making comments,  
and to M. Geck and R. Rouquier for explanations regarding blocks of defect 
1 for Hecke algebras. 
The work of P.E. was  partially supported by the NSF grant 
DMS-0504847. The work of R.B. was supported by the DAPRA grant 
HR0011-04-1-0031 and the NSF grant DMS-0625234.

\section{Rational Cherednik algebras}

\subsection{Definition of rational Cherednik algebras} 
 
Let $\h$ be a finite dimensional vector space over $\CC$, and
$W\subset GL(\h)$ a finite subgroup. 
A reflection in $W$ is an element $s\ne 1$ such that ${\rm
rk}(s-1)=1$. Denote by $S$ the set of reflections in $W$. 
Let $c: S\to \CC$ be a $W$-invariant function. 
For $s\in S$, let $\alpha_s\in \h^*$ be a generator of ${\rm
Im}(s|_{\h^*}-1)$, and $\alpha_s^\vee\in \h$ be the generator of ${\rm
Im}(s|_{\h}-1)$, such that $(\alpha_s,\alpha_s^\vee)=2$. 

\begin{definition} (see e.g. \cite{EG,E1})
The rational Cherednik algebra $H_c(W,\h)$ is the  
quotient of the algebra $\CC W\ltimes T(\h\oplus \h^*)$  
by the ideal generated by the relations
$$
[x,x']=0,\ [y,y']=0,\ [y,x]=(y,x)-\sum_{s\in S}
c_s(y,\alpha_s)(x,\alpha_s^\vee)s,
$$
$x,x'\in \h^*$, $y,y'\in \h$. 
\end{definition}

\vskip .05in 

{\bf Remark.}
In \cite{EG,E1}, rational Cherednik algebras are
defined for (complex) reflection groups $W$, but this assumption plays no
essential role in the theory, and the same definition can be used
for any finite group. In fact, this is a rather trivial
generalization, since any $W$ acting on $\h$ contains a canonical
normal subgroup $W_{\rm ref}$ generated by the complex
reflections in $W$, and one has $H_c(W,\h)=\Bbb C[W]\otimes_{\Bbb
C[W_{\rm ref}]}H_c(W_{\rm ref},\h)$, with natural multiplication. 

\vskip .05in

An important role in the representation theory of rational
Cherednik algebras is played by the element 
\begin{equation}\label{hform}
\bold h=\sum_i x_iy_i+\frac{\dim
\h}{2}-\sum_{s\in S}\frac{2c_s}{1-\lambda_s}s,
\end{equation}
where $y_i$ is a basis of $\h$, $x_i$ the dual basis of $\h^*$,
and $\lambda_s$ is the nontrivial eigenvalue of $s$ in $\h^*$. 
Its usefulness comes from the fact that it satisfies the
identities
\begin{equation}\label{sca}
[\bold h,x_i]=x_i, [\bold h,y_i]=-y_i.
\end{equation}

\subsection{A geometric approach to rational Cherednik algebras}\label{chera}

In \cite{E2}, a geometric point of view on rational Cherednik algebras
is suggested, in the spirit of the theory of D-modules; this
point of view will be useful in the present paper. 
Namely, in \cite{E2}, the algebra $H_c(W,\h)$ is
sheafified over $\h/W$ (as a usual $\mathcal{O}_{\h/W}$-module). This yields a
quasicoherent sheaf of algebras, $H_{c,W,\h}$, such that for any
affine open subset
$U\subset \h/W$, the algebra of sections $H_{c,W,H}(U)$ is 
$\CC[U]\otimes_{\CC[\h]^W}H_c(W,\h)$. 

One of the main ideas of \cite{E2} (see \cite{E2}, Section 2.9) is that the same sheaf can be 
defined more geometrically as follows. Let $\tilde U$ be the preimage of
$U$ in $\h$. Then the algebra $H_{c,W,\h}(U)$ is the algebra of
linear operators on ${\mathcal O}(\tilde U)$ generated by 
${\mathcal O}(\tilde U)$, the group $W$, and Dunkl-Opdam
operators
$$
\partial_a+\sum_{s\in S}\frac{2c_s}{1-\lambda_s}
\frac{\alpha_s(a)}{\alpha_s}(s-1),
$$
where $a\in \h$. 

\subsection{The category ${\mathcal O}_c(W,\h)$}

The algebra $H_c(W,\h)$ contains commutative subalgebras 
$\CC[\h]$ and $\CC[\h^*]$. We define the category ${\mathcal O}_c(W,\h)$  
to be the category of $H_c(W,\h)$-modules
which are finitely generated over $\CC[\h]=S\h^*$ and 
locally finite under the action of $\h$. 
We have a decomposition 
$$
{\mathcal O}_c(W,\h)=\oplus_{\lambda\in \h^*/W}{\mathcal
O}_c(W,\h)_\lambda,
$$
where ${\mathcal O}_c(W,\h)_\lambda$ is the full subcategory 
of those objects of ${\mathcal O}_c(W,\h)$ on which 
the algebra $\CC[\h^*]^W$ acts with generalized eigenvalue
$\lambda$. For convenience, below we will use the notation 
${\mathcal O}_c(W,\h)_\lambda$ for $\lambda\in \h^*$, rather than $\h^*/W$. 

We note that we have a canonical 
equivalence of categories 
\footnote{It is obvious that $H_{c_1}(W_1,\h_1)\otimes
H_{c_2}(W_2,\h_2)=H_{c_1,c_2}(W_1\times W_2,\h_1\oplus \h_2)$, and
this isomorphism defines an equivalence of categories
$$
{\mathcal O}_{c_1,c_2}(W_1\times W_2,\h_1\oplus
\h_2)_{\lambda_1,\lambda_2}\to 
{\mathcal O}_{c_1}(W_1,\h_1)_{\lambda_1}\otimes 
{\mathcal O}_{c_2}(W_2,\h_2)_{\lambda_2}.
$$
In particular, if we take $W_1=W$, $W_2=1$, $\h_1=\h/\h^W$,
$\h_2=\h^W$, this equivalence specializes to the equivalence
$\zeta$. If $W$ acts trivially on $\h$, then $\zeta$ identifies 
the category of $D$-modules on $\h$ with locally nilpotent action
of $y-\lambda(y)$ with the category of vector spaces, which, upon
taking Fourier transforms, is an instance of Kashiwara's lemma.}
$\zeta: {\mathcal O}_c(W,\h)_\lambda
\to {\mathcal O}_c(W,\h/\h^W)_{\lambda}$, defined by the formula 
$$
\zeta(M)=\{v\in M: yv=\lambda(y)v,\ y\in \h^W\}.
$$ 
This implies that the category ${\mathcal O}_c(W,\h)_\lambda$
depends only on the restriction of $\lambda$ to 
the $W$-invariant complement of $\h^W$ in $\h$. 

The most interesting case is $\lambda=0$. 
The category ${\mathcal O}_c(W,\h)_0$ is  
the category of $H_c(W,\h)$-modules
which are finitely generated under $\CC[\h]$ and 
locally nilpotent under the action of $\h$.
This is what is usually called category ${\mathcal O}$;
it is discussed in detail in \cite{GGOR}.  
It is easy to see using equation (\ref{sca}) that
the element $\bold h$ acts locally finitely in any $M\in
{\mathcal O}_c(W,\h)_0$, with finite dimensional 
generalized eigenspaces, and real parts of eigenvalues bounded 
below. 

The most important objects in the category ${\mathcal
O}_c(W,\h)_0$ are the standard modules
$M_c(W,\h,\tau)={\rm Ind}^{H_c(W,\h)}_{W\otimes \CC[\h^*]}\tau$,
where $\tau$ is an irreducible representation 
of $W$ with the zero action of $\h$, 
and their irreducible quotients $L_c(W,\h,\tau)$. 

It is easy to show that the category ${\mathcal O}_c(W,\h)_0$
contains all finite dimensional $H_c(W,\h)$-modules. 

\begin{remark}
We note that the category ${\mathcal O}_c(W,\h)_0$
is analogous to category ${\mathcal O}$ for semisimple Lie
algebras, while the category ${\mathcal O}_c(W,\h)_\lambda$
is analogous to the category of Whittaker modules. 
\end{remark}

\subsection{Completion of rational Cherednik algebras at zero 
and Jacquet functors} 

Jacquet functors for rational Cherednik algebras 
were defined by Ginzburg, \cite{Gi}. Let us recall 
their construction. 

For any $b\in \h$ we can define the completion
$\widehat{H_c}(W,\h)_b$ to be the algebra of sections  
of the sheaf $H_{c,W,\h}$ on the formal neighborhood of the image
of $b$ in $\h/W$. Namely, $\widehat{H_c}(W,\h)_b$ 
is generated by regular functions on the formal neighborhood of
the $W$-orbit of $b$, the group $W$ and Dunkl-Opdam operators.  

The algebra $\widehat{H_c}(W,\h)_b$ inherits from $H_c(W,\h)$ the natural
filtration $F^\bullet$ by order of differential operators,
and each of the spaces $F^n\widehat{H_c}(W,\h)_b$ has a projective limit
topology; the whole algebra is then equipped with the topology of
the nested union (or inductive limit).  

Consider the completion of the rational Cherednik algebra at
zero, $\widehat{H_c}(W,\h)_0$. It naturally contains the algebra
$\CC[[\h]]$. Define the category $\widehat{\mathcal O}_c(W,\h)$
of representations of $\widehat{H_c}(W,\h)_0$ which are 
finitely generated over $\CC[[\h]]$. 

We have a completion functor $\widehat{}: {\mathcal O}_c(W,\h)
\to\widehat{\mathcal O}_c(W,\h)$, defined by 
$$
\widehat{M}=\widehat{H_c}(W,\h)_0\otimes_{H_c(W,\h)}M=
\CC[[\h]]\otimes_{\CC[\h]}M.
$$ 

Also, for $N\in \widehat{\mathcal O}_c(W,\h)$, let $E(N)$ be the space
spanned by generalized eigenvectors of $\bold h$ in $N$. Then
it is easy to see that $E(N)\in \mathcal O_c(W,\h)_0$. 

The following theorem is standard in the 
theory of Jacquet functors. 

\begin{theorem}\label{equi} 
The restriction of the completion functor $\widehat{}$ 
to ${\mathcal O}_c(W,\h)_0$ is an equivalence 
of categories ${\mathcal O}_c(W,\h)_0\to \widehat{\mathcal
O}_c(W,\h)$. The inverse equivalence is given by the functor
$E$. 
\end{theorem} 

\begin{proof} The proof is standard, but we give it for reader's
convenience. 

It is clear that $M\subset \widehat{M}$, 
so $M\subset E(\widehat{M})$ (as $M$ is spanned by generalized 
eigenvectors of $\bold h$). Let us demonstrate the opposite
inclusion. Pick generators $m_1,...,m_r$ of $M$ which are
generalized eigenvectors of $\bold h$ with eigenvalues 
$\mu_1,...,\mu_r$. Let $0\ne v\in E(\widehat{M})$.
Then $v=\sum_i f_im_i$, where $f_i\in \Bbb C[[\h]]$. 
Assume that $(\bold h-\mu)^Nv=0$ for some $N$. Then 
$v=\sum_i f_i^{(\mu-\mu_i)}m_i$, where for $f\in \Bbb C[[\h]]$ we
denote by $f^{(d)}$ the degree $d$ part of $f$. 
Thus $v\in M$, so $M=E(\widehat{M})$. 

It remains to show that $\widehat{E(N)}=N$, i.e. that 
$N$ is the closure of $E(N)$. In other words, 
letting $\mathfrak{m}$ denote the maximal ideal 
in $\Bbb C[[\h]]$, we need to show that the natural map 
$E(N)\to N/\mathfrak{m}^jN$ is surjective for every $j$. 

To do so, note that $\bold h$ preserves the descending filtration of $N$
by subspaces $\mathfrak{m}^jN$. On the other hand, the successive
quotients of these subspaces, $\mathfrak{m}^jN/\mathfrak{m}^{j+1}N$,
are finite dimensional, which implies that $\bold h$ acts locally
finitely on their direct sum ${\rm gr}N$, and moreover each
generalized eigenspace is finite dimensional. Now for each 
$\beta\in \Bbb C$ denote by $N_{j,\beta}$ the generalized
$\beta$-eigenspace of $\bold h$ in $N/{\mathfrak m}^jN$. 
We have surjective homomorphisms $N_{j+1,\beta}\to N_{j,\beta}$, 
and for large enough $j$ they are isomorphisms. This implies that
the map $E(N)\to N/\mathfrak{m}^jN$ is surjective for every $j$,
as desired. 
\end{proof} 
 
\vskip .05in 

{\bf Example.} Suppose that $c=0$. Then Theorem \ref{equi} 
specializes to the well known fact that the category of 
$W$-equivariant local systems on $\h$ with a locally nilpotent
action of partial differentiations is equivalent to the category 
of all $W$-equivariant local systems on the formal neighborhood
of zero in $\h$. In fact, both categories in this case are
equivalent to the category of finite dimensional representations
of $W$.  

\vskip .05in

We can now define the composition functor
${\mathcal J}: {\mathcal O}_c(W,\h)\to {\mathcal O}_c(W,\h)_0$, 
by the formula ${\mathcal J}(M)=E(\widehat{M})$.  
The functor ${\mathcal J}$ is called the Jacquet functor
(\cite{Gi}). 

\subsection{Generalized Jacquet functors} 

\begin{proposition}\label{nilpo} For any $M\in \widehat{\mathcal
O}_c(W,\h)$, a vector $v\in M$ is ${\bold h}$-finite 
if and only if it is $\h$-nilpotent.
\end{proposition}

\begin{proof} The ``if'' part actually holds for any
$H_c(W,\h)$-module $M$. Namely, if $v$ is $\h$-nilpotent then 
consider the finite dimensional space $S\h\cdot v$. 
We prove that $v$ is ${\bold h}$-finite by
induction in the dimension $d$ of this space. 
We can use $d=0$ as the base, so we only need to do the induction
step. The space $S\h\cdot v$ must contain a nonzero vector $u$ such that
$yu=0$ for all $y\in \h$. Let $U\subset M$ be the
subspace of vectors with this property. Formula (\ref{hform})
for ${\bold h}$ implies that ${\bold h}$ acts in $U$ by an element of the group
algebra of $W$, hence locally finitely. So it is sufficient 
to prove that the image of $v$ in $M/<U>$ is ${\bold h}$-finite
(where $<U>$ is the submodule generated by $U$). But this is true
by the induction assumption, as $u=0$ in $M/<U>$. 

To prove the ``only if'' part, assume that $({\bold
h}-\mu)^Nv=0$. Then for any $u\in S^r\h\cdot v$, we
have $({\bold h}-\mu+r)^Nv=0$. But by Theorem \ref{equi},
the real parts of generalized eigenvalues of ${\bold h}$ in $M$
are bounded below. Hence $S^r\h\cdot v=0$ for large enough
$r$, as desired. 
\end{proof} 

According to Proposition \ref{nilpo}, the functor $E$
can be alternatively defined by setting $E(M)$ to be the subspace of
$M$ which is locally nilpotent under the action of $\h$. 

This gives rise to the following generalization of $E$: 
for any $\lambda\in \h^*$ we define the functor 
$E_\lambda: \widehat{\mathcal O}_c(W,\h)\to {\mathcal
O}_c(W,\h)_\lambda$ by setting $E_\lambda(M)$ to be the space of 
generalized eigenvectors of $\CC[\h^*]^W$ in $M$ with eigenvalue
$\lambda$. This way, we have $E_0=E$. 

We can also define the generalized Jacquet functor 
${\mathcal J}_\lambda: {\mathcal O}_c(W,\h)\to {\mathcal
O}_c(W,\h)_\lambda$ by the formula 
${\mathcal J}_\lambda(M)=E_\lambda({\widehat M})$. 
Then we have ${\mathcal J}_0={\mathcal J}$, and the restriction
of ${\mathcal J}_\lambda$ to ${\mathcal
O}_c(W,\h)_\lambda$ is the identity functor. 

\subsection{The duality functors}\label{duali}  

Let $\bc\in \CC[S]^W$ be defined by $\bc(s)=c(s^{-1})$. Then 
we have a natural isomorphism $\gamma: H_\bc(W,\h^*)^{op}\to
H_c(W,\h)$, 
acting trivially on $\h$ and $\h^*$, and 
sending $w\in W$ to $w^{-1}$ (\cite{GGOR}, 4.2). Thus, if $M$ is an
$H_c(W,\h)$-module, then the full dual space $M^*$ 
is naturally an $H_\bc(W,\h^*)$-module, via 
$\pi_{M^*}(a)=\pi_M(\gamma(a))^*$. 

It is clear that the duality functor $*$ defines an equivalence 
between the category ${\mathcal O}_c(W,\h)_0$ and 
$\widehat{\mathcal O}_{\bc}(W,\h^*)^{op}$, and that 
we can define the functor of 
restricted dual $\dagger: {\mathcal O}_c(W,\h)\to 
{\mathcal O}_{\bc}(W,\h^*)^{op}$, given by the formula 
$M^\dagger=E(M^*)$. This functor assigns to $M$ its restricted dual space
under the grading by generalized eigenvalues of $\bold h$. 
It is clear that this functor is 
an equivalence of categories, and $\dagger^2={\rm id}$.

\section{Parabolic induction and restriction functors} 

\subsection{Parabolic subgroups} 

For a point $a$ of $\h$ or $\h^*$, let $W_a$ denote the
stabilizer of $a$ in $W$.
Define a parabolic subgroup of $W$ to be
the stabilizer $W_b$ of a point $b\in \h$.
The set of conjugacy classes
of parabolic subgroups in $W$ 
will be denoted by ${\rm Par}(W)$.

Suppose $W'\subset W$ is a parabolic subgroup, and $b\in \h$ is
such that $W_b=W'$. In this case, we have a natural
$W'$-invariant decomposition
$$
\h=\h^{W'}\oplus (\h^{*W'})^\perp,
$$ and $b\in \h^{W'}$. 
Thus we have a nonempty open set $\h^{W'}_{\reg}$ of all 
$a\in \h^{W'}$ for which $W_a=W'$; this set is nonempty because
it contains $b$. We also have a $W'$-invariant decomposition 
$\h^*=\h^{*W'}\oplus (\h^{W'})^\perp$, and we can define 
the open set $\h^{*W'}_{\reg}$ of all 
$\lambda\in \h^{W'}$ for which $W_\lambda=W'$. 
It is clear that this set is nonempty. 
This implies, in particular, that one can make an alternative
definition of a parabolic subgroup of $W$ as the stabilizer of a
point in $\h^*$. 

\subsection{The centralizer construction}

For a finite group $H$, let $e_H=\frac{1}{|H|}\sum_{h\in H}h$ 
be the idempotent of the
trivial representation in $\Bbb C[H]$. 

If $G\supset H$ are finite groups, 
and $A$ is an algebra containing $\Bbb C[H]$, 
then define the algebra 
$Z(G,H,A)$ to be the centralizer $\End_A(P)$ of $A$ in the right $A$-module
$P={\rm Fun}_H(G,A)$ of $H$-invariant $A$-valued functions on $G$,
i.e. such functions $f: G\to A$ that $f(hg)=hf(g)$. Clearly, 
$P$ is a free $A$-module of rank $|G/H|$, so the algebra 
$Z(G,H,A)$ is isomorphic to ${\rm Mat}_{|G/H|}(A)$, but this
isomorphism is not canonical. 

The following lemma is trivial. 

\begin{lemma}\label{le}
(i) The functor $N\mapsto I(N):=P\otimes_A N={\rm Fun}_H(G,N)$ defines 
an equivalence of categories 
$A-{\rm mod}\to Z(G,H,A)-{\rm mod}$. 
 
(ii) $e_GZ(G,H,A)e_G=e_HAe_H$. 

(iii) $Z(G,H,A)e_GZ(G,H,A)=Z(G,H,A)$ if and only if $Ae_HA=A$. 
\end{lemma} 

\subsection{Completion of rational Cherednik algebras at
arbitrary points of $\h/W$}

The following result is, in essense, a consequence 
of the geometric approach to rational Cherednik algebras,
described in Subsection \ref{chera}. 
It should be regarded as a direct generalization to the 
case of Cherednik algebras of Theorem 8.6 of \cite{L} 
for affine Hecke algebras.

\begin{theorem}\label{comp}
Let $b\in \h$, and $c'$ be the restriction of $c$ to 
the set $S_b$ of reflections in $W_b$. Then 
one has a natural isomorphism 
$$
\theta: \widehat{H_c}(W,\h)_b\to Z(W,W_b,\widehat{H_{c'}}(W_b,\h)_0),
$$
defined by the following formulas. 
Suppose that $f\in P={\rm Fun}_{W_b}(W,\widehat{H_c}(W_b,\h)_0)$.
Then 
$$
(\theta(u)f)(w)=f(wu), u\in W;
$$
for any $\alpha\in \h^*$, 
$$
(\theta(x_\alpha)f)(w)=(x_{w\alpha}^{(b)}+(w\alpha,b))f(w), 
$$
where $x_\alpha\in \h^*\subset H_c(W,\h)$, 
$x_\alpha^{(b)}\in \h^*\subset H_{c'}(W_b,\h)$ are the elements
corresponding to $\alpha$; and for any 
$a\in \h^*$, 
\begin{equation}\label{thetay}
(\theta(y_a)f)(w)=y_{wa}^{(b)}f(w)+\sum_{s\in S: s\notin W_b}
\frac{2c_s}{1-\lambda_s}\frac{\alpha_s(wa)}{x_{\alpha_s}^{(b)}+\alpha_s(b)}
(f(sw)-f(w)).
\end{equation}
where $y_a\in \h\subset H_c(W,\h)$, 
$y_a^{(b)}\in \h\subset H_{c'}(W_b,\h)$. 
\end{theorem}

\begin{proof}
The proof is by a direct computation. 
We note that in the last formula, the fraction 
$\frac{\alpha_s(wa)}{x_{\alpha_s}^{(b)}+\alpha_s(b)}$
is viewed as a power series (i.e., an element of $\CC[[\h]]$),
and that only the entire sum, and not each summand separately, 
is in the centralizer algebra.  
\end{proof} 

\vskip .05in

{\bf Remark.} Let us explain how to see the existence of $\theta$
without writing explicit formulas, and how to guess the formula 
(\ref{thetay}) for $\theta$. 
It is explained in \cite{E2} (see e.g. \cite{E2}, Section 2.9) 
that the sheaf of algebras obtained
by sheafification of $H_c(W,\h)$ over $\h/W$ is generated (on
every affine open set in $\h/W$) by regular functions on $\h$,
elements of $W$, and Dunkl-Opdam operators. Therefore, this
statement holds for formal neighborhoods, i.e., it
is true on the formal neighborhood of the image in $\h/W$ 
of any point $b\in \h$. However, looking at the formula for Dunkl-Opdam
operators near $b$, we see that the summands corresponding to
$s\in S,s\notin W_b$ are actually regular at $b$, so they can be
safely deleted without changing the generated algebra (as all regular functions
on the formal neighborhood of $b$ are included into the system of
generators). But after these terms are deleted, what remains is
nothing but the Dunkl operators for $(W_b,\h)$, which, together
with functions on the formal neighborhood of $b$ and the group
$W_b$, generate the completion of $H_c(W_b,\h)$. This gives a 
construction of $\theta$ without using explicit formulas. 

Also, this argument explains why $\theta$ should be defined by
the formula (\ref{thetay}) 
of Theorem \ref{comp}. Indeed, what this formula does
is just restores the terms with $s\notin W_b$ that have been
previously deleted.

\vskip .05in

The map $\theta$ defines an equivalence of categories
$$
\theta_{*}: \widehat{H_c}(W,\h)_b-{\rm mod}\to
Z(W,W_b,\widehat{H_{c'}}(W_b,\h)_0)-{\rm mod}.
$$

\begin{corollary}\label{whit}
We have a natural equivalence of categories 
$\psi_\lambda: {\mathcal O}_c(W,\h)_\lambda\to 
{\mathcal O}_{c'}(W_\lambda,\h/\h^{W_\lambda})_0$. 
\end{corollary}

\begin{proof} The category ${\mathcal O}_c(W,\h)_\lambda$
is the category of modules over $H_c(W,\h)$ 
which are finitely generated over $\CC[\h]$ and extend by
continuity to the completion of the algebra $H_c(W,\h)$ at $\lambda$. So
it follows from Theorem \ref{comp}
that we have an equivalence ${\mathcal O}_c(W,\h)_\lambda\to 
{\mathcal O}_{c'}(W_\lambda,\h)_0$. Composing this equivalence
with the equivalence 
$\zeta: {\mathcal O}_{c'}(W_\lambda,\h)_0\to 
{\mathcal O}_{c'}(W_\lambda,\h/\h^{W_\lambda})_0$, we obtain
the desired equivalence $\psi_\lambda$.    
\end{proof} 

\begin{remark}
Note that in this proof, we take the completion of $H_c(W,\h)$ 
at a point of $\lambda\in \h^*$ rather than $b\in \h$.  
\end{remark}

\subsection{The completion functor}

Let $\widehat{\mathcal O}_c(W,\h)^b$ be the category 
of modules over $\widehat{H_c}(W,\h)_b$ 
which are finitely generated over $\widehat{\CC[\h]}_b$. 

\begin{proposition}\label{isom}
The duality functor $*$ defines an anti-equivalence of categories 
${\mathcal O}_c(W,\h)_\lambda\to \widehat{\mathcal O}_\bc(W,\h^*)^\lambda$.  
\end{proposition}

\begin{proof}
This follows from the fact (already mentioned above) 
that ${\mathcal O}_c(W,\h)_\lambda$ is the category of 
modules over $H_c(W,\h)$ 
which are finitely generated over $\CC[\h]$ and extend by
continuity to the completion of the algebra $H_c(W,\h)$ at
$\lambda$. 
\end{proof} 

Let us denote the functor inverse to $*$ also by $*$; it is the 
functor of continuous dual (in the formal series topology). 

We have an exact functor of completion at $b$, 
${\mathcal O}_c(W,\h)_0\to \widehat{\mathcal O}_c(W,\h)^b$,
$M\mapsto \widehat{M}_b$. We also have a functor 
$E^b: \widehat{\mathcal O}_c(W,\h)^b\to {\mathcal O}_c(W,\h)_0$
in the opposite direction, sending a module $N$ to 
the space $E^b(N)$ of $\h$-nilpotent vectors in $N$. 

\begin{proposition}\label{adj1}
The functor $E^b$ is right adjoint to the completion functor
$\widehat{}_b$. 
\end{proposition}

\begin{proof} Straightforward.
\end{proof} 

\begin{remark}
Recall that by Theorem \ref{equi}, if $b=0$ then these functors 
are not only adjoint but also inverse to each other. 
\end{remark} 

\begin{proposition}\label{dua}
(i) For $M\in {\mathcal O}_\bc(W,\h^*)_b$, one has
$E^b(M^*)=(\widehat{M})^*$ in ${\mathcal O}_c(W,\h)_0$. 

(ii)  For $M\in {\mathcal O}_c(W,\h)_0$, 
$(\widehat{M}_b)^*=E_b(M^*)$ in ${\mathcal O}_\bc(W,\h^*)_b$. 

(iii) The functors $E_b$, $E^b$ are exact.
\end{proposition}

\begin{proof}
(i),(ii) are straightforward from the definitions. 
(iii) follows from (i),(ii), since the completion functors are
exact.  
\end{proof} 

\subsection{Parabolic induction and restriction functors for rational Cherednik
algebras} 

Theorem \ref{comp} allows us to define analogs of parabolic
restriction functors for rational Cherednik algebras. 

Namely, let $b\in\h$, and $W_b=W'$. 
Define a functor $\Res_b: {\mathcal O}_c(W,\h)_0\to
{\mathcal O}_{c'}(W',\h/\h^{W'})_0$ by the formula
$$
\Res_b(M)=(\zeta\circ E\circ I^{-1}\circ
\theta_*)(\widehat{M}_b).
$$  
 
We can also define the parabolic induction functors
in the opposite direction. Namely, 
let $N\in {\mathcal O}_{c'}(W',\h/\h^{W'})_0$. 
Then we can define the object $\Ind_b(N)\in {\mathcal
O}_c(W,\h)_0$ by the formula 
$$
\Ind_b(N)=(E^b\circ \theta_*^{-1}\circ
I)(\widehat{\zeta^{-1}(N)}_0).
$$

\begin{proposition}\label{propres} 
(i) The functors $\Ind_b$, $\Res_b$ are exact.

(ii) One has $\Ind_b(\Res_b(M))=E^b(\widehat{M}_b)$.
\end{proposition} 

\begin{proof} Part (i) follows from the fact that the
functor $E^b$ and the completion functor $\widehat{}_b$ 
are exact (see Proposition \ref{dua}).
Part (ii) is straightforward from the definition.  
\end{proof}

\begin{theorem}\label{adj2}
The functor $\Ind_b$ is right adjoint to 
$\Res_b$.  
\end{theorem} 

\begin{proof} We have 
$$
\Hom(\Res_b(M),N)=
\Hom((\zeta\circ E\circ I^{-1}\circ \theta_*)(\widehat{M}_b),N)=
$$
$$
\Hom((E\circ I^{-1}\circ \theta_*)(\widehat{M}_b),\zeta^{-1}(N))=
$$
$$
\Hom((I^{-1}\circ \theta_*)(\widehat{M}_b),\widehat{\zeta^{-1}(N)}_0)=
\Hom(\widehat{M}_b,(\theta_*^{-1}\circ
I)(\widehat{\zeta^{-1}(N)}_0))=
$$
$$
\Hom(M,(E^b\circ \theta_*^{-1}\circ
I)(\widehat{\zeta^{-1}(N)}_0))=\Hom(M,\Ind_b(N)).
$$
At the end we used Proposition \ref{adj1}.  
\end{proof} 

\begin{corollary}\label{proj} The functor $\Res_b$ maps projective objects 
to projective ones, and the functor $\Ind_b$ maps injective objects to
injective ones. 
\end{corollary} 

We can also define functors
$\res_\lambda: {\mathcal O}_c(W,\h)_0\to 
{\mathcal O}_{c'}(W',\h/\h^{W'})_0$ and $\ind_\lambda: 
{\mathcal O}_{c'}(W',\h/\h^{W'})_0\to 
{\mathcal O}_c(W,\h)_0$, attached to $\lambda\in
\h^{*W'}_{\reg}$, by 
$$
\res_\lambda:= \dagger\circ \Res_\lambda\circ \dagger, 
\ind_\lambda:= \dagger\circ \Ind_\lambda\circ \dagger, 
$$
where $\dagger$ is as in Subsection \ref{duali}.

\begin{corollary}\label{proj1} 
The functors $\res_\lambda$, 
$\ind_\lambda$ are exact. 
The functor $\ind_\lambda$ is left adjoint to $\res_\lambda$. 
The functor $\ind_\lambda$ maps projective objects 
to projective ones, and the functor $\res_\lambda$ injective objects to
injective ones. 
\end{corollary} 

We also have the following proposition, whose proof is
straightforward. 

\begin{proposition}\label{Jacc}
We have 
$$
\ind_\lambda(N)=({\mathcal J}\circ \psi_\lambda^{-1})(N),
$$
and 
$$
\res_\lambda(M)=(\psi_\lambda\circ E_\lambda)(\widehat{M}),
$$
where $\psi_\lambda$ is defined in Corollary \ref{whit}.
\end{proposition}

\subsection{Some evaluations of the parabolic induction and 
restriction functors}

For generic $c$, the category ${\mathcal O}_c(W,\h)$ is
semisimple, and naturally equivalent to the category $\Rep W$ 
of finite dimensional representations of $W$, via the functor $\tau\mapsto
M_c(W,\h,\tau)$. (If $W$ is a Coxeter group, the exact set of
such $c$ (which are called regular) is known from 
\cite{GGOR} and \cite{Gy}).   

\begin{proposition}\label{compuRI} (i) Suppose that $c$ is generic. 
Upon the above identification, the functors
$\Ind_b$, $\ind_\lambda$ and $\Res_b$, $\res_\lambda$ 
go to the usual induction and restriction
functors between categories $\Rep W$ and $\Rep W'$.  
In other words, we have 
$$
\Res_b(M_c(W,\h,\tau))=\oplus_{\xi\in \widehat{W'}}n_{\tau\xi}M_{c'}(W',\h/\h^{W'},\xi),
$$
and 
$$
\Ind_b(M_{c'}(W',\h/\h^{W'},\xi))=\oplus_{\tau\in \widehat{W}}
n_{\tau\xi}M_{c}(W,\h,\tau),
$$
where $n_{\tau\xi}$ is the multiplicity of occurrence of 
$\xi$ in $\tau|_{W'}$, and similarly for ${\rm res}_\lambda$, 
${\rm ind}_\lambda$.  

(ii) The equations of (i) hold at the level of Grothendieck
groups for all $c$. 
\end{proposition} 

\begin{proof}
Part (i) is easy for $c=0$, and is obtained for generic $c$ 
by a deformation argument. Part (ii) is also obtained 
by deformation argument, taking into account that the functors
$\Res_b$ and $\Ind_b$ are exact and flat with respect to $c$. 
\end{proof} 

\begin{example} Suppose that $W'=1$. Then 
$\Res_b(M)$ is the fiber of $M$ at $b$, while
$\Ind_b(\CC)=P_{KZ}$, the object defined in \cite{GGOR}, which is
projective and injective (see Remark \ref{conhol}). This shows that Proposition
\ref{compuRI} (i) does not hold for special $c$, as $P_{KZ}$ is
not, in general, a direct sum of standard modules.  
\end{example}

\subsection{Dependence of the functor $\Res_b$ on $b$}\label{dep}

Let $W'\subset W$ be a parabolic subgroup. 
In the construction of the functor $\Res_b$, 
the point $b$ can be made a variable which belongs to the open
set $\h^{W'}_{\reg}$. 

Namely, let $\widehat{\h^{W'}_{\reg}}$
be the formal neighborhood of the locally closed set $\h^{W'}_{\reg}$
in $\h$, and let $\pi: \widehat{\h^{W'}_{\reg}}\to \h/W$ be 
the natural map (note that this map is an \'etale covering
of the image with the Galois group $N_W(W')/W'$, where
$N_W(W')$ is the normalizer of $W'$ in $W$). 
Let $\widehat{H_c}(W,\h)_{\h^{W'}_{\reg}}$ be the
pullback of the sheaf $H_{c,W,\h}$ under $\pi$. 
We can regard it as a sheaf of algebras over 
$\h^{W'}_{\reg}$. Similarly to Theorem \ref{comp}
we have an isomorphism 
$$
\theta: \widehat{H_c}(W,\h)_{\h^{W'}_{\reg}}\to 
Z(W,W',\widehat{H_{c'}}(W',\h/\h^{W'})_0)\hat\otimes
D(\h^{W'}_{\reg}),
$$
where $D(\h^{W'}_{\reg})$ is the sheaf of differential
operators on $\h^{W'}_{\reg}$, and $\hat\otimes$ is an
appropriate completion of the tensor product.

Thus, repeating the construction of $\Res_b$, we can define 
the functor 
$$
\Res: {\mathcal O}_c(W,\h)_0\to 
{\mathcal O}_{c'}(W',\h/\h^{W'})_0\boxtimes {\rm
Loc}(\h^{W'}_{\reg}),
$$
where ${\rm Loc}(\h^{W'}_{\reg})$ stands for the category of
local systems (i.e. O-coherent D-modules) on $\h^{W'}_{\reg}$. 
This functor has the property that $\Res_b$ is the fiber of $\Res$ at $b$.
Namely, the functor $\Res$ is defined by the formula 
$$
\Res(M)=(E\circ I^{-1}\circ
\theta_*)(\widehat{M}_{\h^{W'}_{\reg}}),
$$  
where $\widehat{M}_{\h^{W'}_{\reg}}$ is the restriction of the
sheaf $M$ on $\h$ to the formal neighborhood of $\h^{W'}_{\reg}$.

\begin{remark}
If $W'$ is the trivial group, the functor $\Res$
is just the KZ functor from \cite{GGOR}. 
Thus, $\Res$ is a relative version of the KZ functor. 
\end{remark}

Thus, we see that the functor $\Res_b$ does not depend on $b$, up
to an isomorphism. A similar statement is true for the functors 
$\Ind_b$, $\res_\lambda$, $\ind_\lambda$.

\begin{conjecture}\label{bl}
For any $b\in \h,\lambda\in \h^*$ such that $W_b=W_\lambda$, 
we have isomorphisms of functors $\Res_b\cong \res_\lambda$,
$\Ind_b\cong \ind_\lambda$. 
\end{conjecture}

\begin{remark} Conjecture \ref{bl} would imply that 
$\Ind_b$ is left adjoint to $\Res_b$, and 
that $\Res_b$ maps injective objects to injective ones, while
$\Ind_b$ maps projective objects to projective ones. 
\end{remark}

\begin{remark}\label{conhol} If $b$ and $\lambda$ are generic (i.e.,
$W_b=W_\lambda=1$) then the conjecture holds. 
Indeed, in this case the conjecture reduces to showing that 
we have an isomorphism of functors ${\rm Fiber}_b(M)\cong
{\rm Fiber}_\lambda(M^\dagger)^*$ ($M\in {\mathcal O}_c(W,\h)$). 
Since both functors are exact functors
to the category of vector spaces, it suffices to check that 
$\dim {\rm Fiber}_b(M)=\dim {\rm Fiber}_\lambda(M^\dagger)$. 
But this is true because both
dimensions are given by the leading coefficient of the Hilbert
polynomial of $M$ (characterizing the growth of $M$).   
\end{remark}

It is important to mention, however, that
although $\Res_b$ is isomorphic to $\Res_{b'}$ if $W_b=W_{b'}$,
this isomorphism is not canonical. So let us examine the dependence
of $\Res_b$ on $b$ a little more carefully. 

Theorem \ref{compuRI} implies that if $c$ is generic, then 
$$
\Res(M_c(W,\h,\tau))=\oplus_{\xi}M_{c'}(W',\h/\h^{W'},\xi)
\otimes{\mathcal L}_{\tau\xi},
$$
where ${\mathcal L}_{\tau\xi}$ is a local system on $\h^{W'}_{\reg}$ of rank
$n_{\tau\xi}$. Let us characterize the local system 
${\mathcal L}_{\tau\xi}$ explicitly. 

\begin{proposition}\label{tauxi}
The local system ${\mathcal L}_{\tau\xi}$ is given by the
``partial'' KZ connection on
the trivial bundle, with the connection form 
$$
\sum_{s\in S: s\notin W'}\frac{2c_s}{1-\lambda_s}
\frac{d\alpha_s}{\alpha_s}(s-1).
$$
with values in $\Hom_{W'}(\xi,\tau|_{W'})$. 
\end{proposition}

\begin{proof}
This follows immediately from formula (\ref{thetay}).
\end{proof} 

\subsection{Supports of modules}
 
The following two basic propositions are proved in \cite{Gi},
Section 6. We will give different proofs of them, based on the 
restriction functors. 

\begin{proposition}\label{supp} 
Consider the stratification of $\h$ with respect to stabilizers
of points in $W$. Then the support ${\rm Supp}M$ of any object
$M$ of ${\mathcal
O}_c(W,\h)$ in $\h$ is a union of strata of this stratification.  
\end{proposition}

\begin{proof}
This follows immediately from the existence of the flat connection 
along the set of points $b$ with a fixed stabilizer $W'$ on the bundle
$\Res_b(M)$. 
\end{proof} 

\begin{proposition}\label{supp1}
For any irreducible object $M$ in 
${\mathcal O}_c(W,\h)$, ${\rm Supp}M/W$ is an irreducible algebraic variety. 
\end{proposition} 

\begin{proof}
Let $X$ be a component of ${\rm Supp}M/W$. 
Let $M'$ be the subspace of elements of $M$ whose 
restriction to a neighborhood of a generic point 
of $X$ is zero. It is obvious that $M'$ is an 
$H_c(W,\h)$-submodule in $M$. By definition, it is a proper
submodule. Therefore, by the irreducibility of $M$, we have
$M'=0$. Now let $f\in \Bbb C[\h]^W$ be a function that vanishes
on $X$. Then there exists a positive integer 
$N$ such that $f^N$ maps $M$ to $M'$, hence acts by zero 
on $M$. This implies that
${\rm Supp}M/W=X$, as desired. 
\end{proof} 

Propositions \ref{supp} and \ref{supp1} allow us to attach 
to every irreducible module $M\in {\mathcal
O}_c(W,\h)$, a conjugacy class of parabolic subgroups, $C_M\in
Par(W)$, namely, the conjugacy class of the 
stabilizer of a generic point of the support
of $M$. Also, for a parabolic subgroup $W'\subset W$, 
denote by ${\mathcal S}(W')$ the set of points $b\in \h$ whose
stabilizer contains a subgroup conjugate to $W'$.

The following proposition is immediate. 

\begin{proposition}\label{supp2} 
(i) Let $M\in {\mathcal
O}_c(W,\h)_0$ be irreducible. If $b$ is such that $W_b\in C_M$,
then $\Res_b(M)$ is a nonzero finite dimensional module over
$H_{c'}(W_b,\h/\h^{W_b})$.  
 
(ii) Conversely, let $b\in \h$, and $L$ be a finite dimensional
module $H_c(W_b,\h/\h^{W_b})$. 
Then the support of $\Ind_b(L)$ in $\h$ 
is ${\mathcal S}(W_b)$. 
\end{proposition}

Let $FD(W,\h)$ be the set of $c$ for which $H_c(W,\h)$ admits a
finite dimensional representation. 

\begin{corollary}\label{supp3} 
(i) Let $W'$ be a parabolic subgroup of $W$.   
Then ${\mathcal S}(W')$ is the support of some irreducible representation 
from ${\mathcal O}_c(W,\h)_0$ if and only if $c'\in
FD(W',\h/\h^{W'})$.  

(ii) Suppose that $W$ is a Coxeter group. 
Then the category ${\mathcal O}_c(W,\h)_0$ is semisimple 
if and only if $c\notin \cup_{W'\in {\rm Par}(W)}FD(W',\h/\h^{W'})$.
\end{corollary}

\begin{proof}
(i) is immediate from Proposition \ref{supp2}, and 
(ii) follows from (i), since by the combination of results from 
\cite{DJO},\cite{Gy}, and \cite{GGOR}, the category ${\mathcal O}_c(W,\h)_0$ 
is not semisimple if and only if 
there exists a nonzero representation in ${\mathcal O}_c(W,\h)_0$
whose support is not equal to $\h$. 
\end{proof} 

\begin{example}
Let $W=S_n$, $\h=\Bbb C^{n-1}$.  In this case, the set ${\rm Par}(W)$ is 
the set of partitions of $n$. Assume that $c=r/m$, $(r,m)=1$, 
$2\le m\le n$. By a result of \cite{BEG2},
finite dimensional representations of $H_c(W,\h)$ exist 
if and only if $m=n$. Thus the only possible classes $C_M$
for irreducible modules $M$ have stabilizers $S_m\times...\times
S_m$, i.e., correspond to partitions into parts, 
where each part is equal to $m$ or $1$. So there are 
$[n/m]+1$ possible supports for modules, 
where $[a]$ denotes the integer part of $a$.    
\end{example} 

\subsection{Cuspidal numbers} 

Let $W$ be a real  
reflection group, $\h$ its reflection representation.
Let us say that a function $c$ is {\it singular}
if the category ${\mathcal O}_c(W,\h)_0$ is not semisimple. 
It follows from \cite{GGOR,Gy,DJO} that if $c$ is constant and 
$c>0$ then $c$ is singular if and only if 
the polynomial representation $M_c(W,\h,\Bbb C)$ is
reducible. The paper \cite{DJO} determines the
set of singular values. 
In particular, it is shown in \cite{DJO} that constant $c>0$ 
is singular if and only if $c\in \Bbb Q$, and the denominator 
of $c$ divides one of the degrees $d_i$ of $W$.  

In this subsection we assume that $c$ is a constant function. 
Let ${\rm Div}(W,\h)$ be the set of all divisors 
of the degrees $d_i$ of $W$. 
 
Let us say that $d$ is a {\it cuspidal number} for $W$ 
if $d\in {\rm Div}(W,\h)$, but $d\notin {\rm Div}(W',\h)$
for any proper parabolic subgroup $W'\subset W$. 
Thus, constant $c$ with denominator being a cuspidal number 
is a special kind of singular values. 

\begin{proposition}\label{equicon}
The following two conditions on $c$ are equivalent:

(a) The category ${\mathcal O}_c(W,\h)_0$ is not semisimple, 
but any representation $M\in {\mathcal O}_c(W,\h)_0$ 
is either finite dimensional or has full support in $\h$. 

(b) the denominator of $c$, when written as an irreducible
fraction, is a cuspidal number of $W$.  
\end{proposition}

\begin{proof} As we have mentioned, ${\mathcal O}_c(W,\h)_0$
is not semisimple iff the denominator of $c$ divides a degree of
$d_i$ of $W$. Thus, by Corollary \ref{supp3}, condition (a) holds if and
only if the denominator of $c$ divides a degree of $W$, but does
not divide a degree of a proper parabolic subgroup, which proves
the proposition.   
\end{proof}

A basic example of a cuspidal number for any irreducible $W$ is 
the Coxeter number $h$ of $W$, since it is greater than any of
the degrees for parabolic subgroups. Let us call any other 
cuspidal number non-Coxeter, and denote the set of such numbers
$NC(W)$.  

The non-Coxeter cuspidal numbers are found by inspecting tables.  
Let us enumerate them. Classical Weyl groups (of type A,B=C,D)
do not have non-Coxeter cuspidal numbers. Here are the
non-Coxeter cuspidal numbers for other irreducible Coxeter
groups: 
$$
NC(E6)=\lbrace{9\rbrace}, NC(E7)=\lbrace{14\rbrace},
NC(E8)=\lbrace{15,20,24\rbrace}, NC(F4)=\lbrace{8\rbrace},
$$
$$
NC(I(m))=\lbrace{2<d<m: m/d\in \Bbb Z\rbrace},
NC(H3)=\lbrace{6\rbrace}, NC(H4)=\lbrace{12,15,20\rbrace}.
$$

\begin{corollary}\label{fd}
Suppose that $c>0$ and the denominator of $c$, when written as an irreducible
fraction, is a cuspidal number of $W$.
Then the representation $L_c(W,\h,\CC)$ is finite dimensional.
\end{corollary}

\begin{proof} 
Since the denominator of $c$ divides a degree of $W$, 
$c$ is a singular value, and since $c>0$, by \cite{DJO}, the polynomial representation
$M_c(W,\h,\CC)$ is reducible, so $L_c(W,\h,\CC)$ cannot have full
support. So by Proposition \ref{equicon}, 
it is finite dimensional, as desired. 
\end{proof}

\begin{remark}
If $W$ is a Weyl group, this proposition follows from the main
result of \cite{VV}, which states that for $c>0$, 
$L_c(W,\h,\Bbb C)$ is finite dimensional if and only if the
denominator of $c$ is an elliptic number, because every 
cuspidal number is an elliptic
number\footnote{An element $w\in W$ is {\it elliptic} if it has no 
nonzero invariants in the reflection representation $\h$, and is
{\it regular} if it has a regular eigenvector in $\h$. An {\it elliptic
number} is, by definition, the order of an elliptic regular
element (see \cite{VV}).}.
\end{remark}

\begin{remark}
We note that in the case when $W$ is a Weyl group and $d=h$, 
i.e. $c=j/h$, $j\in \Bbb N$, $(j,h)=1$, 
the fact that the representations $L_c(W,\h,\CC)$ are finite
dimensional follows from the work of Cherednik (see
\cite{Ch1}); these are the so-called perfect representations,
of dimension $j^r$, where $r$ is the rank of $W$.  
More precisely, Cherednik works with the true double affine Hecke
algebras ${\mathcal H}_{q,t}$ 
(not with their rational degenerations), but it is known
(\cite{Ch1}) that finite dimensional representations for the two kinds of
algebras have the same structure if $q=e^\hbar$, $t=e^{\hbar c}$,
where $\hbar$ is a formal parameter. 
\end{remark}

Now suppose $c$ is as in Corollary \ref{fd}, and
consider the KZ functor $\O_c(W,\h)_0\to {\rm
Rep}\mathcal{H}_q(W)$, where $q=e^{2\pi ic}$, and 
$\mathcal{H}_q(W)$ is the corresponding finite Hecke algebra. 
Then it follows  from the results of \cite{GGOR} and
Proposition \ref{equicon} that this functor 
kills finite dimensional irreducible modules, 
and sets up a bijection between other irreducible modules (with
full support) and irreducible representations of  
$\mathcal{H}_q(W)$. Thus we get 

\begin{corollary}\label{diffe} 
The number of irreducible finite dimensional representations 
of $H_c(W,\h)$ equals $N(W)-N_q(W)$, where $N(W)$ is the number of irreducible representations
of $W$, and $N_q(W)$ is the number of irreducible representations of
${\mathcal H}_q(W)$.
\end{corollary}

\begin{remark}
It turns out (see \cite{GP}) 
that if the denominator of $c$ is a cuspidal number 
then $N(W)-N_q(W)$ is always 1 or 2, and it is 2 only 
in the cases when $W$ is of type $E_8$
or $H_4$ and $d=15$. In both of these cases, the additional
finite dimensional irreducible representation is the one whose
highest weight is the reflection representation of $W$. 
\end{remark}

\begin{remark}
The results of this subsection can also be found in the latest
version of the paper \cite{Rou}, Section 5.2.4, which appeared
while this paper was being written. 
\end{remark} 

\section{The Gordon-Stafford theorem}

\subsection{Aspherical parameter values}

Let $M$ be a nonzero $H_c(W,\h)$-module. Let us say that 
$M$ is {\it aspherical} if $e_WM=0$. Let $c$ be called 
aspherical if $H_c(W,\h)$ admits an aspherical
representation which belongs to the category ${\mathcal O}_c(W,\h)_0$. 
Let $\Sigma(W,\h)$ be the set of aspherical values.
If $W'\subset W$ is a parabolic subgroup,
then denote by $\Sigma'(W',\h)$ the preimage 
of $\Sigma(W',\h)$ in $\Bbb C[S]^W$ under the restriction map
$c\mapsto c'$.

Let also $FDA(W,\h)$ be the set of $c$ for which $H_c(W,\h)$
admits a finite dimensional aspherical representation. 

\begin{theorem}\label{main}
(i) $c\in \Sigma(W,\h)$ if and only if $H_c(W,\h)e_W H_c(W,\h)\ne
H_c(W,\h)$.

(ii)  We have 
$$
\Sigma(W,\h)=FDA(W,\h)\cup\bigcup_{W'\in {\rm Par}(W)}\Sigma'(W',\h/\h^{W'}).
$$
\end{theorem}

\begin{proof}
(i) This is essentially proved in \cite{BEG1}. Only 
the ``if'' direction requires proof. Let 
$B=H_c(W,\h)/H_c(W,\h)e_W H_c(W,\h)$; we have $B\ne 0$.
Let us regard $B$ as a $(\CC[\h]^W,\CC[\h^*]^W)$-bimodule; 
then it is finitely generated. Thus if $I$ is the maximal 
ideal in $\CC[\h^*]^W$ corresponding to the point $0$, 
then $B/BI\ne 0$. So $B/BI$ is a module from category ${\mathcal
O}_c(W,\h)_0$ which is aspherical. Hence $c$ is aspherical. 

(ii) By Lemma \ref{le} and Theorem \ref{comp} 
if $c\notin \Sigma(W,\h)$ then
$$
H_{c'}(W',\h/\h^{W'})e_{W'}H_{c'}(W',\h/\h^{W'})=H_{c'}(W',\h/\h^{W'}),
$$
which by (i) implies that $c'$ is not aspherical. 

Thus, $\Sigma(W,\h)$ contains the union 
$FDA(W,\h)\cup\bigcup_{W'\in {\rm Par}(W)}\Sigma'(W',\h/\h^{W'})$.
It remains to show that it is also contained in this union. 
To this end, let $c\in \Sigma(W,\h)$. 
Then there exists a module $M\ne 0$ from category ${\mathcal O}_c(W,\h)_0$ such
that $e_WM=0$. If $M$ is finite dimensional, then $c\in
FDA(W,\h)$, and we are done. Otherwise, $M$ must have a nonzero
support in $\h$. Let $b\in \h$ be a nonzero point of this support, and
$M_b=\Res_b(M)$. This is a module from category
${\mathcal O}_{c'}(W_b,\h/\h^{W_b})$, which is killed by
$e_{W_b}$. 
Thus, $c'\in \Sigma(W_b,\h/\h^{W_b})$, and $c\in 
\Sigma'(W_b,\h/\h^{W_b})$, as desired. 
\end{proof}

\begin{corollary}\label{gorsta} If $W=S_n$ and $\h$ its reflection
representation, then $\Sigma(W,\h)$ is the set $Q_n$ of rational
numbers in $(-1,0)$ with denominator $\le n$.  
\end{corollary} 

This is a slight strengthening of the result of Gordon and Stafford \cite{GS} 
who proved that $\Sigma(S_n,\h)\setminus Q_n$ is a 
(finite) set contained in $\frac{1}{2}+ \Bbb Z$. 
It was proved earlier in \cite{DJO}, Theorem 4.9, 
that $\Sigma(S_n,\h)\supset Q_n$. 

\begin{proof}
It follows from the results of \cite{BEG2} 
that 
$$
FDA(S_n,\h)=\lbrace{r/n|-n<r<0, GCD(n,r)=1\rbrace}.
$$ 
Thus the result follows from Theorem \ref{main} immediately by
induction in $n$. 
\end{proof} 

Recall (\cite{BEG1}) that we have translation (or shift) functors 
$$
F: H_c(S_n,\h)-{\rm mod}\to H_{c+1}(S_n,\h)-{\rm mod},
F_*: H_{c+1}(S_n,\h)-{\rm mod}\to H_{c}(S_n,\h)-{\rm mod} 
$$ 
defined by the formulas 
$$
F(V)=H_{c+1}e_-\otimes_{e_-H_{c+1}e_-=e_+H_ce_+}e_+V,\quad
F_*(V)=H_ce_+\otimes_{e_+H_ce_+=e_-H_{c+1}e_-}e_-V,
$$
where we use a shorthand notation $H_c:=H_c(S_n,\h)$, and 
$e_+,e_-$ are the symmetrizer and antisymmetrizer for $S_n$. 

\begin{corollary} If $c\notin Q_n$ then the translation functor 
$F$ is an equivalence of categories. 
\end{corollary} 

This corollary was proved in \cite{GS} for $c\notin
\frac{1}{2}+\Bbb Z$. 

\begin{proof} 
We have $F_*F(V)=H_ce_+V$, and $FF_*(U)=
H_{c+1}e_-U$. There is an automorphism of $H_c$ sending $c$ to
$-c$ and $e_+$ to $e_-$; also $Q_n$ is stable under the
map $c\to -1-c$. This implies that $F_*F(V)=V$, $FF_*(U)=U$, so
$F$ is an equivalence. 
\end{proof} 

\subsection{Aspherical values of $c$ for real reflection groups}

For a general $W$, the determination of the set $\Sigma(W,\h)$ is
an interesting open problem. For instance, let $W$ be a real 
reflection group, $\h$ its reflection representation, 
and $c$ a constant function. 
Let us say that $c$ is {\it strongly singular} if the module 
$L_c(W,\h,{\rm sign})$ is aspherical. It follows
from \cite{DJO}, Theorem 4.9, that $c$ is strongly singular 
if and only if $c=-j/d_i$, where $1<j<d_i-1$, and $d_i$ are the
degrees of the generators in $\CC[\h]^W$. Also, it is clear that
any strongly singular $c$ is aspherical. Thus, for any $i,j$ as
above, $-j/d_i\in \Sigma(W,\h)$. Finally, any aspherical value 
$c\in (-1,0)$ is strongly singular, since for other $c\in
(-1,0)$, the category ${\mathcal O}_c(W,\h)_0$ is semisimple,
\cite{GGOR}, and hence all simple objects have $W$-invariant
vectors (as they coincide with the corresponding Verma modules). 

However (contrary to what was conjectured in the previous version 
of this paper), an aspherical value of $c$ need not belong to $(-1,0)$ 
and may be positive. This was pointed out to us by M. Balagovic and A. Puranik, 
who found an aspherical representation for $c=1/2$ and the Coxeter group of type $H_3$. 
There are also the following examples for type $B_n$.  

\begin{example} Let $W=S_n\ltimes \Bbb Z_2^n$ be the Weyl group of type $B_n$, 
and $\h$ its reflection representation. 
It is shown in \cite{EM} that if $m,l$ are positive integers with 
$ml=n$, and $2c(m-l\pm 1)=1$, then there exists a 
representation $U$ of $H_c(W,\h)$ which is irreducible as a $W$-module.
Namely, the action of $\Bbb Z_2^n$ on $U$ is by $\pm 1$
(depending on the above choice of sign), and as a representation $S_n$,
$U$ is isomorphic to the irreducible representation $\pi_\lambda$, where 
$\lambda$ is the rectangular Young diagram with $m$ columns and $n$ rows. 

In particular, if $m-l\pm 1>0$ but $l>1$, we get an aspherical representation with $c>0$.    
\end{example}

\subsection{Aspherical representations for $S_n$}

Let $W=S_n$, $\h$ be its reflection representation, 
and let $c=-r/m$, $2\le m\le n$, $1\le r\le m-1$
(so $c\in (-1,0)$). 

\begin{proposition}\label{asph} An irreducible representation
$L=L_c(\tau)$ of 
$H_c(W,\h)$ is aspherical if and only if its support is not equal
to $\h$. 
\end{proposition}

\begin{proof}
Suppose that the support of $L$ is $\h$. Then $L|_{\h_{\rm reg}}\ne 0$, 
so $(L|_{\h_{\rm reg}})^W\ne 0$, and hence $L^{W}\ne 0$, so $L$ is not
aspherical. 

Conversely, suppose the support of $L$ is $X\ne \h$, i.e. $X/W$
is an irreducible subvariety
of $\h/W$. Let $b\in X$ be a generic point. In this case, as we
have seen, 
${\rm Res}_b(L)$ is a finite dimensional representation of
$H_c(W_b)$. 
Since $-1<c<0$, and $W_b$ is a product of symmetric groups, we
see that 
${\rm Res}_b(L)$ is aspherical. 

Let $X'\subset X$ be the open set of points
with stabilizer conjugate to $W_b$. Because ${\rm Res}_b(L)$ is
aspherical,
we have $(L|_{X'})^W=0$. But since $L$ is irreducible, the map
$L\to L|_{X'}$ 
is injective, so $L^W=0$, and $L$ is aspherical.    
\end{proof} 

\begin{corollary}\label{spheal} For $-1<c<0$, the category ${\mathcal O}_{\rm
spherical}$ for the spherical subalgebra 
$e_W H_c(W,\h)e_W$ is equivalent to the category of finite
dimensional representations of 
the Hecke algebra ${\mathcal H}_q(W)$, where $q=e^{2\pi ic}$. 
\end{corollary}

\begin{proof} 
According to Proposition \ref{asph} and the paper \cite{GGOR}, 
both categories are equvalent to ${\mathcal O}/{\mathcal O}_{\rm
tor}$, 
where ${\mathcal O}_{\rm tor}$ is the Serre subcategory of
objects which are torsion as modules over 
${\mathcal C}[\h]$. 
\end{proof} 

\begin{corollary}\label{dj} For $c=-r/m$ as above, 
$L_c(\lambda)$ is aspherical if and only if the corresponding
partition $\lambda$ is not $m$-regular, i.e., if it contains some part at
least $m$ times. 
\end{corollary} 

\begin{proof} Let $q=e^{2\pi ic}$, a primitive $m$-th root of
unity. Recall from \cite{DJ} that for every partition $\lambda$ we have
the Specht module $S_\lambda$ 
over the Hecke algebra ${\mathcal H}_q:={\mathcal H}_q(S_n)$ and
its quotient $D_\lambda$, which is either simple (if $\lambda$ is
$m$-regular) or zero (if not), and this gives an enumeration,
without repetitions, of irreducible representations of ${\mathcal
H}_q$. Moreover, it is known (\cite{DJ}, theorem 7.6)
that all the composition factors of $S_\lambda$ are 
$D_\mu$ with $\mu\ge \lambda$ (in the dominance ordering), and 
the multiplicity of $D_\lambda$ in $S_\lambda$ 
(when $D_\lambda$ is nonzero) is $1$. 

Let us say that a simple object $L$ of ${\mathcal O}_c(W,\h)$ is
thin if $KZ(L)=0$, otherwise let us say that it is thick.  
By Proposition \ref{asph}, $L_c(\lambda)$ is aspherical if and
only if it is thin. 

Our job is to show that $L_c(\lambda)$ is thick iff 
$\lambda$ is $m$-regular, and in this case
$KZ(L_c(\lambda))=D_\lambda$. This follows from the paper 
\cite{Rou} (Section 5), but we give a proof here for reader's convenience. 

Let $N(\lambda):=\frac{n(n-1)}{2}-c(\lambda)$, where 
$c(\lambda)$ is the content of $\lambda$. 
Note that if $\nu>\lambda$ then $N(\nu)<N(\lambda)$. 
We prove that the statement holds for $N(\lambda)\le k$, by
induction in $k$. 

If $k=0$ then $\lambda=(n)$ and the statement is clear. 
Now suppose the statement is known for $k-1$ and let 
us prove it for $k$. 

By \cite{GGOR}, Theorem 5.14, the KZ functor is exact and maps a simple object
either to zero or to a simple object, so for any $\mu$, $KZ(L_c(\mu))=0$ if
$L_c(\mu)$ is thin, and $KZ(L_c(\mu))=D_{\nu(\mu)}$ for some
$\nu=\nu(\mu)$ if $L_c(\mu)$ is thick. 
Also, by \cite{GGOR}, Corollary 6.10, $KZ(M_c(\mu))=S_\mu$.
This means that $\nu(\mu)\ge \mu$ for all $\mu$. 
 
Let $\lambda$ be such that $N(\lambda)=k$. If $L_c(\lambda)$ is
thin then by the above argument, $KZ(M_c(\lambda))$ has
composition factors $D_\mu$ with $\mu>\lambda$. Since
$KZ(M_c(\lambda))=S_\lambda$,
this implies that $S_\lambda$ 
has composition factors $D_\mu$ with $\mu>\lambda$. By Theorem
7.6 of \cite{DJ}, this implies
that $\lambda$ is not $m$-regular. On the other hand, if
$L_c(\lambda)$ is thick, then $\nu(\lambda)$ is
$m$-regular, and by the induction  
assumption, if $\nu(\lambda)>\lambda$ then 
$D_{\nu(\lambda)}$ also equals $KZ(L_c(\nu(\lambda)))$, so two
irreducible modules have the same nonzero image 
under the KZ functor, which contradicts Theorem 5.14 of
\cite{GGOR}. Thus, $\nu(\lambda)=\lambda$, and $\lambda$ is
$m$-regular. This completes the induction step. 
\end{proof}

\begin{remark}
Note that it is well known (and easy to see) that the generating function
for the number of $m$-regular partitions is 
$$
f_m(q)=\frac{\phi(q^m)}{\phi(q)},
$$
where $\phi$ is the Euler function, 
$$
\phi(q)=\prod_{n\ge 1}(1-q^n). 
$$
\end{remark}

\begin{remark} 
A. Okounkov and the first author conjectured that 
the number of aspherical representations in ${\mathcal
O}_c(S_n,\h)_0$ for each $n$ is given
by the rank of the residue 
of the connection describing the equivariant small quantum 
cohomology of the Hilbert scheme of $\CC^2$ at 
$q=-e^{2\pi i c}$ (\cite{OP}). According to \cite{OP}, 
this residue is proportional to the operator
$$
\sum_{s\ge 1}\alpha_{-ms}\alpha_{ms}
$$
on the degree $n$ part of the Fock representation of the
Heisenberg Lie algebra, with commutation relations  
$[\alpha_i,\alpha_j]=\delta_{i,-j}$.
Thus, the conjecture follows from Corollary \ref{dj}.  
Indeed, by Corollary \ref{dj} and the previous remark, 
the conjecture is equivalent to saying that the kernel of this operator 
has character $f_m(q)$, which is obvious, since this
kernel is the space of polynomials of $\alpha_{-i}$, $i\ge 1$, 
such that $i$ is not divisible by $m$. 

One can also observe that the eigenvalues of this residue are
proportional to the codimensions of supports of the 
modules in ${\mathcal O}_c(S_n,\h)_0$.
\end{remark}

\subsection{The simplicity of the spherical subalgebra for 
$-1<c<0$ in type $A$}

In \cite{BEG1}, it is shown that the algebra $H_c(W,\h)$ 
is simple if and only if $c$ is not a singular value, and in this
case $H_c(W,\h)$ is Morita equivalent to its spherical subalgebra
$e_W H_c(W,\h)e_W$. This implies that if $c$ is not singular, the spherical
subalgebra is simple, while if $c$ is singular and $H_c(W,\h)$ is
Morita equivalent to $e_W H_c(W,\h) e_W$, it is not.
However, it turns out that when $H_c$ and $e_W H_c e_W$ are not Morita 
equivalent, it can happen that $e_W H_c e_W$ is simple. 
Namely, we have the following result. 

\begin{theorem} The spherical subalgebra
$A_c(n):=e_{S_n}H_c(S_n,\h)e_{S_n}$ is simple for $c\in (-1,0)$. 
\end{theorem}

\begin{proof}
Let $I\subset A_c(n)$ be a proper two-sided ideal, and consider 
the $A_c(n)$-bimodule $M:=A_c(n)/I\ne 0$. 
Let us regard $M$ as a module over \linebreak $\Bbb C[x_1,...,x_n]^{S_n}
\otimes \Bbb C[y_1,...,y_n]^{S_n}$ 
by acting with the first factor on the left side and with the second
one on the right side. Obviously, $M$ is finitely generated. 
Let $Z\subset (\Bbb C^n/S_n)^2$ be the support of $M$. 
Then $Z$ is a nonempty closed subvariety of $(\Bbb C^n/S_n)^2$.
Let $p_1,p_2: Z\to \Bbb C^n/S_n$ be the two projections. 
Let $b\in \Bbb C^n$ be a point such that $p_2^{-1}(S_nb)$ is nonempty. 
Then the fiber $M_b$\footnote{We abuse notation by denoting the
orbit of $b$ under $S_n$ also by $b$.} of $M$ is a nonzero left
$A_c(n)$-module, finitely generated over 
$\Bbb C[x_1,...,x_n]^{S_n}$, on which symmetric polynomials of $y_i$ act
locally finitely, with generalized eigenvalue $b$.
So it belongs to the category $e_{S_n}{\mathcal
O}_c(S_n,\h)_\lambda$. By Corollary \ref{whit} and 
Proposition \ref{asph}, this implies that the support of 
$M_b$ as a $\Bbb C[x_1,...,x_n]^{S_n}$-module is the entire $\Bbb
C^n/S_n$. Thus, $p_2^{-1}(b)=\Bbb C^n/S_n$ if it is nonempty. 
Similarly one proves that $p_1^{-1}(b)=\Bbb C^n/S_n$ if it is
nonempty. Thus, we find that $Z=(\Bbb C^n/S_n)^2$, which implies
that $I=0$.      
\end{proof} 

\subsection{A strengthening of the result of \cite{BFG}}

In this subsection we apply Corollary \ref{dj} 
to enhance the main result of
\cite{BFG}.

\subsubsection{A modification of a result of \cite{GG}}

We start with a slightly modified version of Theorem 6.6.1 of \cite{GG}. 
Let $K$ be a field of characteristic zero, $\g={\mathfrak
{gl}}_n(K)$, and $D(\g)$ the algebra of differential operators on
$\g$. Let $c$ be an indeterminate, and $V_c$ be the
representation of $\g$ on the space of ``functions'' of the form 
$(x_1....x_n)^c f(x_1,...,x_n)$, where $f$ is a Laurent polynomial
with coefficients in $K[c]$ of total degree zero (namely, $\g$
acts through its projection to ${\mathfrak
{sl}}_n(K)$). Let ${\rm Ann}V_c$ be the annihilator of $V_c$ in 
$U(\g)[c]$, and $J_c=D(\g)[c]{\rm ad}({\rm Ann V_c})$, where
${\rm ad}: U(\g)\to D(\g)$ is the adjoint action. It is shown in 
\cite{EG} that one has a filtration preserving homomorphism 
$$
\Phi_c: (D(\g)/D(\g)J_c)^\g\to A_c(n),
$$
where $A_c(n)$ is the spherical rational Cherednik algebra 
with parameter $c$ being an indeterminate
(here $D(\g)$ is filtered by order of differential operators, and 
$A_c(n)$ inherits the filtration from the full Cherednik algebra
$H_c(n)$, which is filtered by $\deg(\h^*)=\deg(S_n)=0$, $\deg(\h)=1$). 

\begin{theorem}\label{661}
The associated graded map ${\rm gr}\Phi_c$, and hence $\Phi_c$
itself, are isomorphisms.  
\end{theorem}

The difference between Theorem \ref{661} and Theorem 6.6.1 of
\cite{GG} is that in \cite{GG}, $c$ is any fixed 
element of $K$, while here $c$ is an indeterminate. 
However, this distinction is inessential, and 
the proof of Theorem \ref{661} is parallel to the proof of
Theorem 6.6.1 of \cite{GG}.

\subsubsection{The case of positive characteristic} 

Now fix a positive integer $n$ and let $k$ be an algebraically closed field of
sufficiently large (compared to $n$) prime characteristic $p$.

Let $X$ denote the Hilbert scheme of $n$ points on the plane ${\mathbb
A}^2_k$. As above, let $A_c=A_c(n)$ denote
the spherical rational Cherednik algebra 
with parameter $c$.

For $c\in {\mathbb F}_p$ an Azumaya algebra ${\mathbb A}_c$ of rank $p^{2n}$
on $X^{(1)}$ was defined in \cite{BFG}; here
the superscript $^{(1)}$ denotes the 
Frobenius twist. Furthermore, one has the following strengthened
version of the second statement of Theorem 7.2.1 of \cite{BFG}.

\begin{theorem}\label{721} For $p\gg n$, 
the algebra of global sections $\Gamma({\mathbb A}_c)$
is canonically isomorphic to $A_c$. 
\end{theorem}

Theorem 7.2.1 of \cite{BFG} claims that this statement holds for
(reduction mod $p$ of) any rational $c$ and $p>d=d(c)$, where $d(c)$ is a constant
depending on $c$. The proof of Theorem \ref{721} is similar to
the proof of Theorem 7.2.1 of \cite{BFG}, using 
Theorem \ref{661}, which had not been known when \cite{BFG}
appeared. 

Theorem \ref{721} implies that we have the functor
$R\Gamma:D^b(Coh(X^{(1)},{\mathbb A}_c))
\to D^b(A_c-mod^{fg})$ where $Coh(X^{(1)},{\mathbb A}_c)$ is the category of
coherent sheaves of ${\mathbb A}_c$-modules,
and $A_c-mod^{fg}$ is the category of finitely generated modules over $A_c$.
We say that ${\mathbb A}_c$ is {\em derived affine}
if this functor is an equivalence.

\begin{corollary}
For $p\gg n$ the Azumaya algebra ${\mathbb A}_c$ is derived affine if and
only if the inequality $c\ne -\frac{r}{m}$ holds in ${\mathbb F}_p$
for all integers $r,\, m$ such that $0<r<m\leq n$.
\end{corollary}

\begin{proof} The results of \cite{BFG} 
show that ${\mathbb A}_c$ is derived affine if
and only if $A_c$ is Morita equivalent to
the full rational Cherednik algebra $H_c$, i.e. if and only if $H_c e H_c =
H_c$, where $e=e_{S_n}$.
Corollary \ref{dj} 
implies that over a characteristic zero field the last equality
holds exactly when $c\ne -\frac{r}{m}$ for $r,m$ as above.
It follows that the same is true over a field of positive characteristic $p$
for almost all $p$.
\end{proof} 

\newpage

\section{Appendix: Reducibility of the polynomial representation 
of the degenerate double affine Hecke algebra}

\vskip .05in 

\centerline{Pavel Etingof}

\vskip .05in 

\subsection{Introduction}

In this appendix (which has been previously posted as
arXiv:0706.4308) we determine the values of parameters $c$ for which the
polynomial representation of the degenerate double affine Hecke
algebra (DAHA), i.e. the trigonometric Cherednik algebra, is
reducible. Namely, we show that $c$ is a reducibility point 
for the polynomial representation of the trigonometric 
Cherednik algebra for a root system $R$ if and only if it is a 
reducibility point for the {\it rational}
Cherednik algebra for the Weyl group of some  
root subsystem $R'\subset R$
of the same rank given by (one step of) the
well known Borel-de Siebenthal algorithm, \cite{BdS}
(i.e., by deleting a vertex from the extended Dynkin diagram of
$R$).\footnote{It is known from the Borel-de Siebenthal theory 
that any maximal rank root subsystem is obtained by repeating
this process several times; however, the root subsystems $R'$ 
appearing in this appendix are the ones obtained by just one step
of the process; clearly, this contains all the maximal proper 
root subsystems, which correspond to the case where the label 
of the removed vertex is a prime number.}   

This generalizes to the trigonometric case the 
result of \cite{DJO}, where the reducibility points are found for
the rational Cherednik algebra. Together with the result of
\cite{DJO}, our result gives an explicit list of reducibility 
points in the trigonometric case. 

We emphasize that our result is contained in the recent previous work 
of I. Cherednik \cite{Ch2}, where reducibility points 
are determined for nondegenerate DAHA. Namely, the techniques 
of \cite{Ch2}, based on intertwiners, work equally well in 
the degenerate case. In fact, outside of roots of unity, 
the questions of reducibility of the polynomial representation 
for the degenerate and nondegenerate DAHA are equivalent (see
e.g. \cite{VV}, 2.2.4), and thus our result is equivalent to that of 
\cite{Ch2}. However, our proof is quite different from that in
\cite{Ch2}; it is based on the geometric approach 
to Cherednik algebras developed in \cite{E2},
and thus clarifies the results of \cite{Ch2} from a geometric
point of view. In particular, we explain that our result 
and its proof can be generalized 
to the much more general setting of Cherednik algebras for  
any smooth variety with a group action.   
 
We note that in the non-simply laced case, it is not true 
that the reducibility points for $R$ are the same in the
trigonometric and rational settings. In the trigonometric
setting, one gets additional reducibility points, which 
arise for type $B_n$, $n\ge 3$, $F_4$, and
$G_2$, but not for $C_n$. This phenomenon was discovered by
Cherednik (in the $B_n$ case, see \cite{Ch3}, Section 5);
in \cite{Ch2}, he gives a complete list of additional
reducibility points. At first sight, this list looks somewhat 
mysterious; here we demystify it, by interpreting it in 
terms of the Borel - de Siebenthal classification of 
equal rank embeddings of root systems. 

The result of this appendix is a manifestation of the general
principle that the representation theory of the trigonometric
Cherednik algebra (degenerate DAHA) for a root system $R$ reduces
to the representation theory of the {\it rational} Cherednik 
algebra for Weyl groups of root subsystems $R'\subset R$ obtained
by the Borel-de Siebental algorithm. 
This principle is the ``double'' analog of a similar
principle in the representation theory of affine Hecke algebras, 
which goes back to the work of Lusztig \cite{L}, in which it is
shown that irreducible representations of the affine Hecke algebra 
of a root system $R$ may be described in terms of irreducible 
representations of 
the degenerate affine Hecke algebras for Weyl groups 
of root subsystems $R'\subset
R$ obtained by the Borel - de Siebenthal algorithm. 
We illustrate this principle at the end of the 
note by applying it to finite
dimensional representations of trigonometric Cherednik algebras.  

{\bf Acknowledgements.} The author is very grateful to
I. Cherednik for many useful discussions, and for sharing the
results of his work \cite{Ch2} before its publication. 
The author also thanks G. Lusztig, M. Varagnolo, 
and E. Vasserot for useful discussions. 
The work of the author was  partially supported by the NSF grant
DMS-0504847.

\subsection{Preliminaries}

\subsubsection{Preliminaries on root systems}\label{prel}

Let $W$ be an irreducible Weyl group, $\h$ its (complex)
reflection representation, and $L\subset \h$ a $\Bbb Z$-lattice invariant
under $W$. 

For each reflection $s\in W$, let $L_s$ be the intersection of
$L$ with the $-1$-eigenspace of $s$ in $\h$, and let
$\alpha_s^\vee$ be a generator of $L_s$. Let $\alpha_s$ be the
element in $\h^*$ such that $s\alpha_s=-\alpha_s$, and
$(\alpha_s,\alpha_s^\vee)=2$. Then we have
$$
s(x)=x-(x,\alpha_s)\alpha_s^\vee,\ x\in \h.
$$

Let $R\subset \h^*$ be the collection of vectors $\pm \alpha_s$,
and $R^\vee\subset \h$ the collection of vectors
$\pm\alpha_s^\vee$. It is well known that $R,R^\vee$ are mutually
dual reduced root systems.  Moreover, we have $Q^\vee\subset
L\subset P^\vee$, where $P^\vee$ is the coweight lattice, and
$Q^\vee$ the coroot lattice.

Consider the simple complex Lie group $G$ with root system $R$,
whose center is $P^\vee/L$. The maximal torus of $G$ can be
identified with $H=\h/2\pi iL$ via the exponential map. 

For $g\in H$, let $C_\g(g)$ be the centralizer of $g$ in
$\g:={\rm Lie}(G)$.  Then $C_\g(g)$ is a reductive subalgebra of
$\g$ containing $\h$, and its Weyl group is the stabilizer $W_g$
of $g$ in $W$.

Let $\Sigma\subset H$ be the set of elements whose centralizer
$C_\g(g)$ is semisimple (of the same rank as $\g$). 
$\Sigma$ can also be defined as the set of point strata for the
stratification of $H$ with respect to stabilizers. It is well
known that the set $\Sigma$ is finite, and the Dynkin diagram of
$C_\g(g)$ is obtained from the extended Dynkin diagram of $\g$ by
deleting one vertex (the Borel-de Siebenthal
algorithm). Moreover, any Dynkin diagram obtained in this way
corresponds to $C_\g(g)$ for some $g$.

\subsubsection{The degenerate DAHA}

Let $W,L,H$ be as in subsection \ref{prel}. A {\it reflection
hypertorus} in $H$ is a connected component $T$ of the fixed set
$H^s$ for a reflection $s\in W$. Let $c$ be a conjugation
invariant function on the set of reflection hypertori. 
Denote by $\mathcal S$ the set of reflection hypertori. 
For $T\in {\mathcal S}$, denote by $s_T$ the corresponding
reflection, and by $\chi_T$ the affine linear map $H\to \Bbb C^*$ 
such that $\chi_T^{-1}(1)=T$. Let $H_{reg}$ denote the complement
of reflection hypertori in $H$. 

\begin{definition}(Cherednik, \cite{Ch1}) The degenerate DAHA
$H_c(W,H)$ attached to $W,H$ is the algebra generated inside 
$\Bbb C[W]\ltimes D(H_{reg})$ by polynomial functions on
$H$, the group $W$, and trigonometric Dunkl operators
$$
\partial_a+\sum_{T}c(T)\frac{d\chi_T(a)}{1-\chi_T}(s_T-1),
$$
\end{definition}
 
Using the geometric approach of \cite{E2}, which attaches a 
Cherednik algebra to any smooth affine algebraic variety with a 
finite group action, the degenerate DAHA can also be defined as 
the Cherednik algebra $H_{1,c}(W,H)$ attached to the variety $H$
with the action of the finite group $W$.  

Note that this setting includes the case of non-reduced root systems. 
Namely, in the case of a non-reduced root system the function $c$   
may take nonzero values on reflection hypertori which 
don't go through $1\in H$. 

\subsection{The results} 

\subsubsection{The main results} 

The degenerate DAHA has a polynomial representation $M=\Bbb C[H]$
on the space of regular functions on $H$. We would like to
determine for which $c$ this representation is reducible.

Let $g\in \Sigma$. Denote by $c_g$ the restriction of the
function $c$ to reflections in $W_g$; that is, for $s\in W_g$,
$c_g(s)$ is the value of $c$ on the (unique) hypertorus 
$T_{g,s}$ passing through $g$ and fixed by $s$. 

\begin{remark}
If $c(T)=0$ unless $T$ contains $1\in H$ (``the reduced
case''), then $c$ can be regarded as a function of
reflections in $W$, and $c_g$ is the usual restriction of $c$ to
reflections in $W_g$. 
\end{remark} 

Denote by 
${\rm Red}(W,\h)$ the set of $c$ at which the polynomial 
representation $M_c(W,\h,\Bbb C)$ of the rational Cherednik
algebra $H_c(W,\h)$ is reducible. These
sets are determined explicitly in \cite{DJO}.
Denote by ${\rm
Red}_g(W,L)$ the set of $c$ such that $c_g\in {\rm Red}(W_g,\h)$.

Our main result is the following. 

\begin{theorem}\label{irr}
The polynomial representation $M$ of $H_c(W,H)$ is reducible if
and only if $c\in \cup_{g\in \Sigma}{\rm Red}_g(W,L)$. 
\end{theorem} 

The proof of this theorem is given in the next subsection. 

\begin{corollary}\label{irr1} 
If $c$ is a constant function (in particular, if $R$ is simply
laced), then the polynomial representation $M$ of $H_c(W,H)$ 
is reducible if and only if so is the polynomial representation
of the rational Cherednik algebra $H_c(W,\h)$, i.e. iff 
$c=j/d_i$, where $d_i$ is a degree of $W$, and $j$ is a positive
integer not divisible by $d_i$.  
\end{corollary}

\begin{proof}
The result follows from Theorem \ref{irr}, the result of
\cite{DJO}, and the well known fact\footnote{This fact is proved as follows. 
Let $P_W(t)$ be the Poincar\'e polynomial of $W$; 
so 
$$
P_W(t)=\prod_i \frac{1-t^{d_i(W)}}{1-t},
$$  
where $d_i(W)$ are the degrees of $W$. 
Then by Chevalley's theorem, $P_W(t)/P_{W'}(t)$ is a polynomial
(the Hilbert polynomial of the generators of 
the free module $\Bbb C[\h]^{W'}$ over $\Bbb C[\h]^W$).  
So, since the denominator vanishes at a root of unity
of degree $d_i(W')$, so does the numerator, which implies the
statement.} that for any subgroup
$W'\subset W$ generated by reflections, every degree of $W'$
divides some degree of $W$. 
\end{proof} 

However, if $c$ is not a constant function, 
the answer in the trigonometric case may differ from 
the rational case, as explained below.  

\subsubsection{Proof of Theorem \ref{irr}}

Assume first that the polynomial representation $M$ is
reducible. Then there exists a nonzero proper submodule $I\subset M$,
which is an ideal in $\Bbb C[H]$. This ideal defines a subvariety
$Z\subset H$, which is $W$-invariant; it is the support 
of the module $M/I$. It is easy to show using the results of
\cite{E2} (parallel to Proposition \ref{supp} of the present
paper) that $Z$ is a union of strata 
of the stratification of $H$ with respect to stabilizers. 
In particular, since $Z$ is closed, it contains a stratum which
consists of one point $g$. Thus $g\in \Sigma$. 
Consider the formal completion $\widehat{M}_g$ of $M$ at $g$. 
As follows from \cite{E2} (parallel to Section 3 of the present paper), 
this module can be viewed as a module over the formal completion
$\widehat{H}_{c_g}(W_g,\h)_0$ of the rational Cherednik algebra
of the group $W_g$ at $0$, and it has a nonzero proper  
submodule $\widehat{I}_g$. Thus, $\widehat{M}_g$ is reducible,
which implies (by taking nilpotent vectors under $\h^*$)
that the polynomial representation $\bar M$ over
$H_{c_g}(W_g,\h)$ is reducible, hence $c_g\in {\rm Red}(W_g,\h)$,
and $c\in {\rm Red}_g(W,L)$. 

Conversely, assume that $c\in {\rm Red}_g(W,L)$, and thus 
$c_g\in {\rm Red}(W_g,\h)$. Then the polynomial representation
$\bar M$ of $H_{c_g}(W_g,\h)$ is reducible. 
This implies that the completion $\widehat{M}_g=\widehat{\Bbb
C[H]}_g$ is a reducible 
module over $\widehat{H}_{c_g}(W_g,\h)_0$, i.e. it contains a
nonzero proper submodule (=ideal) $J$. Let $I\subset \Bbb C[H]$
be the intersection of $\Bbb C[H]$ with $J$. Clearly, 
$I\subset M$ is a proper submodule (it does not contain $1$). 
So it remains to show that it is nonzero. 
To do so, denote by $\Delta$ a regular function on $H$
which has simple zeros on all the reflection hypertori.  
Then clearly $\Delta^n\in J$ for large enough $n$, 
so $\Delta^n\in I$. Thus $I\ne 0$ and the theorem is proved. 

\subsubsection{Reducibility points in the non-simply laced case}

In this subsection we will consider the reduced non-simply laced
case, i.e. the case of root systems of type $B_n,C_n$, $F_4$, and
$G_2$. In this case, $c$ is determined by two numbers $k_1$ and
$k_2$, the values of $c$ on reflections for long and short roots,
respectively. 

The set ${\rm Red}(W,\h)$ 
is determined for these cases in \cite{DJO}, as the union of the
following lines (where $l\ge 1$, $u=k_1+k_2$, and $i=1,2$). 

$B_n=C_n$: 
$$
2jk_1+2k_2=l,\ l\ne 0\text{ mod }2, j=0,...,n-1,
$$
and 
$$
jk_1=l,\ (l,j)=1,\ j=2,...,n. 
$$

$F_4$: 
$$
2k_i=l,\ 2k_i+2u=l,\ l\ne 0\text{ mod }2; 
3k_i=l,\ l\ne 0\text{ mod }3; 
$$
$$
2u=l, 4u=l,\ l\ne 0\text{ mod }2; 
$$
$$
6u=l,\ l=1,5,7,11\text{ mod }12.
$$

$G_2$: 
$$
2k_i=l,\ l\ne 0\text{ mod }2;\ 3u=l,\ l\ne 0\text{ mod }3.
$$

By using Theorem \ref{irr}, we determine that 
the polynomial representation in the trigonometric case is
reducible on these lines and also on the following additional
lines: 

$B_n$, $n\ge 3$:
$$
(2p-1)k_1=2q, n/2<q\le n-1, p\ge 1, (2p-1,q)=1.
$$

$F_4$:
$$
6k_1+2k_2=l, 4k_1=l,\ l\ne 0\text{ mod }2. 
$$

$G_2$: 
$$
3k_1=l,\ l\ne 0\text{ mod }3. 
$$

In the $C_n$ case, we get no additional lines. 

Note that exactly the same list of additional reducibility points
appears in \cite{Ch2}.

\begin{remark} As explained above, 
the additional lines appear from particular equal rank
embeddings of root systems. Namely, the additional lines for $B_n$ appear from 
the inclusion $D_n\subset B_n$. 
The two series of additional lines for $F_4$ appear from the embeddings
$B_4\subset F_4$ and $A_3\times A_1\subset F_4$, respectively. 
Finally, the additional lines for $G_2$ appear from the embedding
$A_2\subset G_2$.
\end{remark}

\subsubsection{Generalizations}

Theorem \ref{irr} can be generalized, with essentially the same proof, to the
setting of any smooth variety with a group action, as defined in
\cite{E2}. 

Namely, let $X$ be a smooth algebraic variety, and $G$
a finite group acting faithfully on $X$. Let $c$ be a
conjugation invariant function on the set of pairs $(g,Y)$,
where $g\in G$, and $Y$ is a connected component of $X^g$ which
has codimension 1 in $X$. Let $H_{1,c,0,X,G}$ be
the corresponding sheaf of Cherednik algebras defined in
\cite{E2}. We have the polynomial representation ${\mathcal O}_X$
of this sheaf. 

Let $\Sigma\in X$ be the set of points with maximal stabilizer,
i.e. points whose stabilizer is bigger than that of nearby
points. Then $\Sigma$ is a finite set. For $x\in X$, let $G_x$ be the
stabilizer of $x$ in $G$; it is a finite subgroup of $GL(T_xX)$. 
Let $c_x$ be the function of reflections in $G_x$ defined by 
$c_x(g)=c(g,Y)$, where $Y$ is the reflection hypersurface 
passing through $x$ and fixed by $g$ pointwise. 
Let ${\rm Red}_x(G,X)$ be the set of $c$ such that 
$c_x\in {\rm Red}(G_x,T_xX)$ (where, as before, ${\rm Red}(G_x,T_xX)$
denotes the set of values of parameters $c$ for which 
the polynomial representation of the rational Cherednik
algebra $H_c(G_x,T_xX)$ is reducible). 

Then we have the following theorem, whose statement and proof are direct
generalizations of those of Theorem \ref{irr} (which is obtained
when $G$ is a Weyl group and $X$ a torus). 

\begin{theorem}\label{irr2}
The polynomial representation ${\mathcal O}_X$ of $H_{1,c,0,X,G}$
is reducible if and only if $c\in \cup_{x\in \Sigma}{\rm
Red}_x(G,X)$.  
\end{theorem}

Note that this result generalizes in a straightforward way to the case
when $X$ is a complex analytic manifold, and $G$ a discrete
group of holomorphic transformations of $X$. 

\subsection{Finite dimensional representations of the
degenerate double affine Hecke algebra}\footnote{The contents of
this subsection 
arose from a discussion of the author with M. Varagnolo and E. Vasserot.}

Another application of the approach of this appendix is a description
of the category of finite dimensional representations of the 
degenerate DAHA in terms of categories of finite dimensional
representations of rational Cherednik algebras. 
Namely, let $FD(A)$ denote the category of finite dimensional
representations of an algebra (or sheaf of algebras) $A$. 
Then in the setting of the previous subsection we have the
following theorem (see also Proposition 2.22 of \cite{E2}). 

Let $\Sigma'$ be a set of representatives of $\Sigma/G$ in
$\Sigma$.

\begin{theorem}\label{fd1}
One has 
$$
FD(H_{1,c,0,X,G})=\oplus_{x\in \Sigma'}FD(H_{c_x}(G_x,T_xX)).
$$
\end{theorem}

\begin{proof}
Suppose $V$ is a finite dimensional representation of 
$H_{1,c,0,X,G}$. Then the support of $V$ is a union of finitely
many points, and these points must be strata of the
stratification of $X$ with respect to stabilizers, so 
they belong to $\Sigma$. This implies that $V=\oplus_{\xi\in
\Sigma/G} V_\xi$, where $V_\xi$ is supported on the orbit $\xi$. 
Taking completion of the Cherednik algebra at $\xi$, we 
can regard the fiber $(V_\xi)_x$ for $x\in \xi$ as a module over the rational
Cherednik algebra $H_{c_x}(G_x,T_xX)$ (as follows from 
\cite{E2} and the main results of the present paper). 
In this way, $V$ gives rise to an
object of $\oplus_{x\in \Sigma'}FD(H_{c_x}(G_x,T_xX))$.

This procedure can be reversed; this implies the theorem. 
\end{proof} 

\begin{corollary}\label{trig}
One has 
$$
FD(H_c(W,H))=\oplus_{g\in \Sigma/W}FD(H_{c_g}(W_g,\h)).
$$
\end{corollary}

\begin{remark}
Recall that a representation of $H_c(W,H)$ is said to be
spherical if it is a quotient of the polynomial representation. 
It is clear that the categorical equivalence of Corollary
\ref{trig} preserves sphericity of representations (in both
directions). This implies that the results of the paper
\cite{VV}, which classifies spherical finite-dimensional
representations of the rational Cherednik algebras, in fact
yield, through Corollary \ref{trig}, the classification 
of spherical finite dimensional representations of 
degenerate DAHA, and hence of nondgenerate DAHA outside of roots
of unity. We note that the general classification of finite
dimensional representations of Cherednik algebras
outside of type A remains an open problem.  
\end{remark}


\begin{thebibliography}{999}

\bibitem[BdS]{BdS} A. Borel and J. de Siebenthal, Les sous-groupes fermcs de rang
maximum des groupes de Lie clos. Comment Math. Helv., 23 (1949), 200-221.

\bibitem[BEG1]{BEG1} Y. Berest, P. Etingof, V. Ginzburg:
Cherednik algebras and differential operators on quasi-invariants.
Duke Math J. 118 (2003), 279--337,
math.QA/0111005.

\bibitem[BEG2]{BEG2}
Yu. Berest, P. Etingof, V. Ginzburg, Finite dimensional
representations of rational Cherednik algebras,
Int. Math. Res. Not. 2003, no. 19, 1053-1088.

\bibitem[BFG]{BFG} 
Bezrukavnikov, R., Finkelberg, M.; Ginzburg, V., Cherednik
algebras and Hilbert schemes in characteristic $p$. 
With an appendix by Pavel Etingof.  Represent. Theory  10  (2006), 254--298.

\bibitem[Ch1]{Ch1} I. Cherednik, 
Double affine Hecke algebras, 
London Mathematical Society Lecture Note Series, 319,
Cambridge University Press, Cambridge, 2005.

\bibitem[Ch2]{Ch2} 
I. Cherednik, Non-semisimple Macdonald polynomials, 
arXiv:0709.1742.

\bibitem[Ch3]{Ch3} Cherednik, I., Irreducibility of perfect 
representations of double affine Hecke algebras.  
Studies in Lie theory,  79--95, Progr. Math., 243, 
Birkh\"auser Boston, Boston, MA, 2006.

\bibitem[DJO]{DJO} 
C.F. Dunkl, M.F.E. de Jeu and E.M. Opdam, 
Singular polynomials for finite reflection groups, 
Trans. Amer. Math. Soc. 346 (1994), 237-256.

\bibitem[DJ]{DJ} R. Dipper and G. D. James, Representations of
Hecke algebras of the general linear groups, Proc. London
Math. Soc. 52 (1986), 20-52.

\bibitem[E1]{E1} P. Etingof, Calogero-Moser systems and
representation theory, Zurich lectures in advanced mathematics, 
European Mathematical Society, Zurich, 2007, 
arXiv:math/0606233.

\bibitem[E2]{E2} P. Etingof,
 Cherednik and Hecke algebras of varieties with a finite group action,
math.QA/0406499.

\bibitem[EG]{EG} 
P. Etingof, V. Ginzburg, 
Symplectic reflection algebras, Calogero-Moser space, and
deformed Harish-Chandra homomorphism.  Invent. Math.  147
(2002),  no. 2, 243--348.

\bibitem[Gi]{Gi} 
V. Ginzburg, On primitive ideals.  
Selecta Math. (N.S.)  9  (2003),  no. 3, 379--407. 

\bibitem[GG]{GG} W. L. Gan, V. Ginzburg, 
Almost-Commuting Variety, D-Modules, and Cherednik
Algebras, International Math Research Papers, 2006, p.1-54,
math/0409262. 

\bibitem[GGOR]{GGOR} V. Ginzburg, N. Guay, E. Opdam, R. Rouquier, 
On the category ${\mathcal O}$ for rational Cherednik algebras.
Invent. Math.  154  (2003),  no. 3, 617--651. 

\bibitem[GP]{GP} M. Geck, G. Pfeiffer, 
Characters of finite Coxeter groups and 
Iwahori-Hecke algebras,
London Mathematical Society Monographs. 
New Series, 21. The Clarendon Press, 
Oxford University Press, New York, 2000.

\bibitem[GS]{GS}
I. Gordon, J. T. Stafford, {\it Rational Cherednik algebras and 
Hilbert schemes,}  Adv. Math.  {\bf 198}  (2005),  no. 1, 222--274.

\bibitem[Gy]{Gy}
A. Gyoja, Modular representation theory over a ring of
higher dimension with applications to Hecke algebras.  J. Algebra
174  (1995),  no. 2, 553--572.

\bibitem[L]{L} G. Lusztig, 
Affine Hecke algebras and their graded version.
J. Amer. Math. Soc.  2  (1989),  no. 3, 599--635.

\bibitem[EM]{EM}
P. Etingof; S. Montarani,
 Finite dimensional representations of symplectic 
reflection algebras associated to wreath products,
Represent. Theory 9 (2005), 457-467. 

\bibitem[MS]{MS} D. Milicic, W. Soergel,  
The composition series of modules induced from 
Whittaker modules.  Comment. Math. Helv.  72  (1997),  no. 4, 503--520.

\bibitem[OP]{OP} A. Okounkov, R. Pandharipande, 
Quantum cohomology of the Hilbert scheme of points in the plane,
math/0411210. 

\bibitem[Rou]{Rou} R. Rouquier, q-Schur algebras and
complex reflection groups, I, math/0509252.

\bibitem[VV]{VV}
M. Varagnolo, E. Vasserot,
Finite dimensional representations of DAHA and 
affine Springers fibers : the spherical case,
arXiv:0705.2691
\end{thebibliography}
\end{document}